\documentclass[11pt,amsfonts]{article}
\usepackage{graphicx}
\usepackage{latexsym}
\usepackage{amssymb}
\usepackage{amsmath}
\usepackage{enumerate}
\usepackage{color}
\usepackage{layout}
\newtheorem{prop}{Proposition}
\newtheorem{lemma}{Lemma}

\newtheorem{corollary}{Corollary}
\newtheorem{theorem}{Theorem}
\newtheorem{remark}{Remark}

\def\real{{\mathord{{\rm I\kern-2.8pt R}}}}        
\def\inte{{\mathord{{\rm I\kern-2.8pt N}}}}

\def\sZZ{{\rm Z\kern-2.8ptem{}Z}}

\def\z{{\mathchoice
		{\sZZ}
		{\sZZ}
		{\rm Z\kern-0.30em{}Z}
		{\rm Z\kern-0.25em{}Z} }}
\def\sQQ{{\kern 0.27em \vrule height1.45ex width0.03em depth0em
		\kern-0.30em \rm Q}}
\def\qu{{\mathchoice
		{\sQQ}
		{\sQQ}
		{\kern 0.225em \vrule height1.05ex width0.025em depth0em \kern-0.25em \rm Q}
		{\kern 0.180em \vrule height0.78ex width0.020em depth0em \kern-0.20em \rm Q}
}}
\def\sCC{{\kern 0.27em \vrule height1.45ex width0.03em depth0em
		\kern-0.30em \rm C}}
\def\complex{{\mathchoice
		{\sCC}
		{\sCC}
		{\kern 0.225em \vrule height1.05ex width0.025em depth0em \kern-0.25em \rm C}
		{\kern 0.180em \vrule height0.78ex width0.020em depth0em \kern-0.20em \rm C}
}}

\newcommand{\N}{\mathbb{N}}

\newcommand{\ba}{\begin{array}}
	\newcommand{\ea}{\end{array}}
\newcommand{\be}{\begin{equation}}
	\newcommand{\ee}{\end{equation}}
\newcommand{\bea}{\begin{eqnarray}}
	\newcommand{\eea}{\end{eqnarray}}
\newcommand{\beaa}{\begin{eqnarray*}}
	\newcommand{\eeaa}{\end{eqnarray*}}

\newcommand{\eps}{\varepsilon}

\def\z{\zeta}

\font\tenmath=msbm10 \font\sevenmath=msbm7 \font\fivemath=msbm5
\newfam\mathfam \textfont\mathfam=\tenmath
\scriptfont\mathfam=\sevenmath \scriptscriptfont\mathfam=\fivemath
\def\math{\fam\mathfam}

\def \={{\buildrel {\rm (law)} \over =}}
\def \N{{\math N}}

\def\qed{ \hfill \vrule width.25cm height.25cm depth0cm\smallskip}

\newcommand{\basa}{\begin{assumption}}
	\newcommand{\easa}{\end{assumption}}

\newcommand{\bas}{\begin{assum}}
	\newcommand{\eas}{\end{assum}}

\newcommand{\ignore}[1]{}
\begin{document}
	
	\renewcommand{\thefootnote}{\fnsymbol{footnote}}
	
	\renewcommand{\thefootnote}{\fnsymbol{footnote}}

	\title{Multidimensional Stein method and quantitative asymptotic independence}
	
	\author{Ciprian A. Tudor \vspace*{0.2in} \\
		CNRS, Universit\'e de Lille \\
		Laboratoire Paul Painlev\'e UMR 8524\\
		F-59655 Villeneuve d'Ascq, France.\\
		\quad ciprian.tudor@univ-lille.fr\\
		\vspace*{0.1in} }
	
	\maketitle

	\begin{abstract}
		If $ \mathbb{Y}$ is a random vector in $\mathbb{R} ^{d}$, we  denote by $ P_{\mathbb{Y}}$ its probability distribution. Consider a random variable $X$ and a $d$-dimensional random vector $\mathbb{Y}$. Inspired by \cite{Pi}, we develop a multidimensional  Stein-Malliavin calculus  which allows to measure the Wasserstein distance between the law $ P_{ (X, \mathbb{Y})}$ and the probability distribution $ P_{Z}\otimes P_{ \mathbb{Y}}$, where $Z$ is a Gaussian random variable. That is, we give estimates, in terms of the Malliavin operators, for the distance between the law of the random  vector $(X, \mathbb{Y})$ and the law of the vector $(Z, \mathbb{Y})$, where $Z$ is Gaussian and independent of $ \mathbb{Y}$.
		Then we focus on the particular case of random vectors in Wiener chaos and we give an asymptotic version of this result. In this situation, this variant of the Stein-Malliavin calculus has strong and unexpected consequences. 	Let $ (X_{k}, k\geq 1)$ be a sequence of random variables in the $p$th Wiener chaos ($p\geq 2$), which converges in law, as $k\to \infty$, to the Gaussian distribution $ N(0, \sigma ^{2})$.  Also consider  $(\mathbb{Y}_{k}, k\geq 1)$ a $d$-dimensional random sequence converging in $L ^{2}(\Omega)$, as $k\to \infty$, to an arbitrary random vector $\mathbb{U}$ in $\mathbb{R}^{d}$ and assume that the two sequences are asymptotically uncorrelated. We prove that, under very light assumptions on $\mathbb{Y}_{k}$, we have the joint convergence of $ ((X_{k}, \mathbb{Y}_{k}), k\geq 1)$ to $ (Z, \mathbb{U})$ where $  Z\sim N(0, \sigma ^{2})$ is independent of $\mathbb{U}$. These assumptions are automatically satisfied when the components of the vector $\mathbb{Y}_{k}$ belong to a finite sum of Wiener chaoses or when $ \mathbb{Y}_{k}=Y$ for every $k\geq 1$, where $\mathbb{Y} $ belongs to the Sobolev-Malliavin space $\mathbb{D} ^{1,2}$.
	\end{abstract}
	
	\vskip0.3cm
	
	{\bf 2010 AMS Classification Numbers:}  60F05,60G15,60H05,60H07.

	\vskip0.3cm
	
	
	\section{Introduction}
	
	The Stein's method constitutes a collection of mathematical techniques that allow to give quantitative  bounds for  the distance   between the probability distributions of  random variables. It has been initially introduced in the paper \cite{Stein} and then developed by many authors. We refer, among many others to the monographs and surveys \cite{CS}, \cite{Re},  \cite{Ross}, \cite{Stein1} for a detailed description of this method. Of particular interest is the situation when one random variable is Gaussian, but the cases of other target distributions have been analyzed in the literature. 
	
	A more recent theory is the so-called Stein-Malliavin calculus which combines the Stein's method with the techniques of the Malliavin calculus. The first work in this direction is \cite{NP1} (see \cite{NP-book} for a more detailed exposition) and since, numerous authors extended, refined or applied this theory. In this theory, the bounds obtained for the distance between the law of an arbitrary random variable and the target distribution are given in terms of the Malliavin operators. 
	
	The starting point of the  Stein's method for normal approximation is the following observation: $ Z\sim N(0, \sigma ^{2})$ with $\sigma >0$ if and only if 
	\begin{equation*}
		\sigma ^{2} \mathbf{E} f'(Z)- \mathbf{E} Z f(Z) =0
	\end{equation*}  
	for every absolutely continuous  function $ f: \mathbb{R} \to \mathbb{R}$ such that $ \mathbf{E} \vert f'(Z)\vert <\infty$. Then, one can think that if a random variable $X$ has the property that $	\sigma ^{2} \mathbf{E} f'(X)- \mathbf{E} X f(X) $  is close to zero  for a large class of functions $f$, then the probability distribution of $X$ should be close to $ N(0, \sigma ^{2})$.  From this observation, the whole Stein's theory has been constructed, leading to various bounds for the distance between the probability law of the random variable $X$ and the normal distribution $ N(0, \sigma ^{2})$. 
	
	In this work, we deal with a variant of this method recently developed in the  reference \cite{Pi} that allows to measure the distance between the components of a random vector $(X_{1}, X_{2})$, where $ X_{1} \sim N(0, \sigma ^{2})$ and $X_{2}$ has an arbitrary distribution. The nice observation made in \cite{Pi} is that  $X_{1} \sim N(0, \sigma ^{2})$  and  $X_{1}$ is independent of $X_{2}$ if and only if 
	\begin{equation*}
		\sigma ^{2} \mathbf{E}\partial _{x_{1}}f(X_{1}, X_{2})-\mathbf{E} X_{1}f(X_{1}, X_{2})=0
	\end{equation*} 
	for a large class of differentiable functions $ f: \mathbb{R} ^{2} \to \mathbb{R}$. We denoted by $\partial _{x_{1}}f$ the partial derivative of $f$ with respect to its first variable.  As in the standard Stein's method, one follows the intuition that if some random vector $ (X_{1}, X_{2})$ satisfies that $	\sigma ^{2} \mathbf{E}\partial _{x_{1}}f(X_{1}, X_{2})-\mathbf{E}X_{1}f(X_{1}, X_{2})$ is close to zero, then $X_{1}$ should be close in law to $ Z\sim N(0, \sigma ^{2})$ and $ P_{(X_{1}, X_{2})}$ should be close to $ P_{Z}\otimes P_{X_{2}}$. By combining this idea with Malliavin calculus, in \cite{Pi} one gives bounds for the Wasserstein distance between  $P_{(X_{1}, X_{2})}$  and  $ P_{X_{1}}\otimes P_{X_{2}}$ in terms of the Malliavin operators. 
	
	Our purpose is, in a first step, to generalize the above idea by considering random vectors of arbitrary dimension. This extension of the Stein's method combined with Malliavin calculus allows to obtain the following estimate: if $ X \in \mathbb{D} ^{1,2}$ and $ \mathbb{Y}= (Y_{1},..., Y_{d})$ is such that $ Y_{j}\in \mathbb{D} ^{1,2}$ for all $j=1,...,d$, then (we denote by $d_{W}$ the Wasserstein  distance  and $Z\sim N(0, \sigma ^{2})$)
	\begin{equation}\label{28i-1}
		d_{W}\left( P _{ (X, \mathbb{Y})}, P_{Z} \otimes P_{\mathbb{Y}}\right) \leq C\left[ \mathbf{E} \left| \sigma ^{2}-\langle D(-L) ^{-1} X, DX\rangle_{H} \right| + \mathbf{E} \sum_{j=1}^{d} \left| \langle D(-L) ^{-1} X, DY_{j} \rangle_{H} \right| \right], 
	\end{equation}
	with $C>0$. We denoted by $D, L$ the Malliavin derivative and the Ornstein-Uhlenbeck operator with respect to an isonormal process $(W(h), h\in H)$, where $(H, \langle \cdot, \cdot\rangle_{H} ) $ is a real and separable Hilbert space. 
	
	Then, we focus on the particular case of sequences of random variables belonging to a Wiener chaos and we give asymptotic-type results.  We will here show that the convergence of a sequence of multiple stochastic integrals to the Gaussian law has other strong and unexpected consequences. Let $H$ be an Hilbert space and let  $I_{p}$ denote the multiple integral of order $p\geq 1$ with respect to an isonormal process $(W(h), h\in H)$. Assume that $p\geq 2$ is an integer number  and for every $k\geq 1$, $X_{k}=I_{p} (f_{k})$ where $ f_{k}\in H ^{\otimes p}$ are symmetric functions.  Suppose that 
	\begin{equation*}
		X_{k} \to ^{(d) } _{k\to \infty} Z \sim N(0, \sigma ^{2}),
	\end{equation*}
	where $ \sigma>0$ and $"\to ^{(d)} "$ stands for the convergence in distribution. Then the following facts hold true: 
	\begin{itemize}
		\item If $ \mathbb{Y}=(Y_{1},...,Y_{d})$ is a $d$-dimensional random vector with components in the Malliavin-Sobolev space $\mathbb{D} ^ {1,2}$ and $X_{k}, \mathbb{Y}$ are asymptotically uncorrelated \\ (i.e. $\mathbf{E} X_{k}  Y_{j}\to _{k\to \infty}0$ for every $j=1,...,d$), then 
		\begin{equation*}
			(X_{k}, \mathbb{Y})\to ^{(d)} _{k\to \infty} (Z',  \mathbb{Y}),
		\end{equation*}
		with $Z'\sim N(0, \sigma ^{2})$ independent of $ \mathbb{Y}$. 
		\item  Let  $ \left( \mathbb{Y}_{k}= (Y_{1,k},...,Y_{d,k}), k\geq 1\right)$ be  a sequence of random vectors such that each component belongs to the  sum of the first $q$th  Wiener chaoses with $q\leq p$ and $\mathbb{Y}_{k} \to ^{(d)} \mathbb{U}$ ($\mathbb{U}$ is an arbitrary random vector). Then, if $X_{k}, \mathbb{Y}_{k}$ are asymptotically uncorrelated (i.e. for every $j=1,...,d$, $\mathbf{E}X_{k}Y_{j,k}\to _{k\to \infty}0$), then
		\begin{equation*}
			(X_{k}, \mathbb{Y}_{k})\to ^{(d)} _{k\to \infty} (Z', \mathbb{U}),
		\end{equation*}
		where $Z'\sim N(0, \sigma ^{2})$ and $Z',\mathbb{U}$ are independent. 
		
		\item  Let  $ \left( \mathbb{Y}_{k}= (Y_{1,k},...,Y_{d,k}), k\geq 1\right)$ be  a sequence of random vectors such that each component belongs to $\mathbb{D} ^{1,2}$ and satisfies an additional (pretty natural) condition (assumption (\ref{cc4}) in Theorem \ref{tt2}).  Suppose that  $\mathbb{Y}_{k} \to  \mathbb{U}$ in $ L ^{2}(\Omega)$, with $\mathbb{U}$ is an arbitrary $d$-dimensional random vector, and $X_{k}, \mathbb{Y}_{k}$ are asymptotically uncorrelated. Then 
		\begin{equation*}
			(X_{k}, \mathbb{Y}_{k})\to ^{(d)} _{k\to \infty} (Z', \mathbb{U}),
		\end{equation*}
		where $Z'\sim N(0, \sigma ^{2})$ and $Z', \mathbb{U}$ are independent. 
		
		\item If $(Y_{k}, k\geq 1)$ is random sequence in the $q$th Wiener chaos with $q>p$ which converges only in law to $U$, then the joint convergence of $((X_{k}, Y_{k}), k\geq 1)) $ to $(Z, U)$ with $Z, U$ independent does not hold. See the counter-example in Section \ref{counter}.
	\end{itemize}
	These findings may have direct consequences to statistics and limit theorems since many estimators can be expressed as multiple stochastic integrals (see e.g. \cite{T}). The main idea of the proof consists in combining the Fourth Moment Theorem with the multidimensional Stein-Malliavin bound (\ref{28i-1}), and it also involves some interesting technical lemmas (Lemmas \ref{llkey1} and \ref{llkey1}), which may have their own interest.  Let us emphasize that the assumption $p\geq 2$ is crucial. When $p=1$, we cannot expect to have results as those listed above. Indeed, take $X =I_{1}(h)$ with $h\in H, \Vert h\Vert =1$, so $X\sim N(0,1)$. Then $Y= I_{1}(h) ^{2}-1= I_{2} (h^{\otimes 2})$ is an element of the second Wiener chaos, but $X$ and $Y$ are not independent (see e.g. the independence criterion in \cite{UZ}).
	
	We organized the paper as follows. In Section 2, we develop  in a multidimensional context the variant of the Stein-Malliavin calculus introduced in \cite{Pi}. Section 3 contains the statement of our main result concerning the asymptotic independence on Wiener chaos and a short discussion around it and its consequences. Section 4 contains the proof of the main result, which is detailed into several steps. In Section 5 we included several applications of our theory, while Section 6 is the the appendix where we present the basic tools needed throughout our work.

	\section{Multidimensional Stein method}\label{sec2}
	In this paragraph, we generalize the variant of the Stein's method introduced in Section 5 of \cite{Pi} to any dimension $d\geq 1$. Then, we combine it with the techniques of the Malliavin calculus in order to obtain the estimate (\ref{28i-1}).

	\subsection{The method}
	
	The basis of the Stein's method consists in the definition of the Stein's operator and of the Stein's equation. For the normal approximation, the standard operator is
	\begin{equation*}
		\mathcal{L}f(x) = \sigma ^{2} f'(x)-xf(x), \hskip0.3cm x\in \mathbb{R},
	\end{equation*}
	which acts on suitable differentiable functions $f: \mathbb{R} \to \mathbb{R}$. This operator satisfies $\mathbf{E} \mathcal{L} f(Z)=0$  for every $f: \mathbb{R}\to \mathbb{R}$ differentiable with $\mathbf{E} \vert f'(Z) \vert < \infty$ if  and only if $Z\sim N(0, \sigma ^{2})$. The corresponding Stein's equation is
	\begin{equation*}
		\mathcal{L} f(x)= \mathbf{E} h(x)-\mathbf{E}h(Z), \hskip0.3cm x\in \mathbb{R},
	\end{equation*}
	where $h:\mathbb{R} \to \mathbb{R}$ is a given function such that $\mathbf{E}\vert h(Z)\vert <\infty$. The idea of the Stein's method is to find a solution $f_{h}$ to the Stein's equation with nice properties and to use it in order to obtain estimates for $\mathbf{E} h(X)- \mathbf{E} h(Z)$ for an arbitrary random variable $X$.
	
	We follow the same line in a multidimensional context. Now, the purpose is not the normal approximation but to quantify the distance between the probability distribution of a random vector $ (X, \mathbb{Y}) $ and the random vector $(Z, \mathbb{Y}) $ where $Z$ is a centered Gaussian random variable with variance $\sigma ^{2}$ and it is independent of $\mathbb{Y}$.

	Let us consider the operator $\mathcal{N}$ given by
	\begin{equation}
		\label{N}
		\mathcal{N} f (x, \mathbf{y})= \sigma ^{2} \partial_{x} f (x, \mathbf{y})-x f( x,\mathbf{y}), \hskip0.3cm x \in \mathbb{R}, \mathbf{y} \in \mathbb{R} ^{d},
	\end{equation}
	where $\partial _{x_{1}}f$ denotes the partial derivative of $f$ with respect to its first variable. The operator $\mathcal{N}$ acts on the set of differentiable functions $f: \mathbb{R} ^{d+1} \to \mathbb{R}$.
	
	Recall that if  $\mathbb{Y}$ is a random vector, we denote by $P_{\mathbb{Y}}$ its probability distribution. The following two lemmas show that the operator (\ref{N}) characterizes the law of $X$ and the independence of $X$ and $\mathbb{Y} $. The material from this section is inspired from Section 5 in \cite{Pi}.
	
	\begin{lemma}
		Assume $X\sim N(0, \sigma ^{2})$ and $X$ is independent of the random vector $ \mathbb{Y}$. Then  $ \mathbf{E}\mathcal{N}f( X, \mathbb{Y})= 0$ for all $f:\mathbb{R} ^{d+1} \to \mathbb{R} $  differentiable with $\mathbf{E} \vert \partial _{x }f(X, \mathbb{Y}) \vert <\infty$.
	\end{lemma}
	
	\begin{proof}By the standard Stein method, for all $\mathbf{y}   \in \mathbb{R}^{d} $,
		\begin{equation*}
			\sigma ^{2} \mathbf{E} \partial _{x} f ( X, \mathbf{y})=\mathbf{E} Xf(X, \mathbf{y})
		\end{equation*}
		or
		\begin{equation*}
			\sigma ^{2} \int_{\mathbb{R}} \partial _{x} f (x, \mathbf{y}) dP _{X}( x)= \int_{\mathbb{R}} x f (x, \mathbf{y}) d P_{ X}(x).
		\end{equation*}
		Let us integrate with respect to the probability measure $P_{\mathbb{Y}}$. We have (the use of Fubini's theorem is based on Lemma 2.1 in \cite{Ross})
		\begin{eqnarray*}
			&&\sigma ^{2} \int_{\mathbb{R} ^{d}} \left( \int_{\mathbb{R}} \partial _{x} f (x, \mathbf{y}) dP _{X}( x)	\right) dP_{ \mathbb{Y}}(\mathbf{y}) \\
			&=&\sigma ^{2} \int_{\mathbb{R} ^{d+1}}  \partial _{x} f (x, \mathbf{y}) dP _{ X} (x) \otimes  dP_{ \mathbb{Y}}(\mathbf{y}) \\
			&=&\sigma ^{2} \int_{\mathbb{R} ^{d+1}}  \partial _{x} f (x, \mathbf{y}) dP _{ (X, \mathbb{Y})} (x, \mathbf{y}) =\sigma ^{2} \mathbf{E} \partial _{ x} f ( X, \mathbb{Y}),
		\end{eqnarray*}
		where we used the independence of $ X$ and $ \mathbb{Y}$ for the first equality on the above  line. Similarly,
		\begin{eqnarray*}
			&& \int_{\mathbb{R} ^{d}}\left( \int_{\mathbb{R}} x f (x, \mathbf{y} ) d P_{ X}(x)\right)dP_{  \mathbb{Y}}(\mathbf{y}) \\
			&=& \int_{\mathbb{R} ^{d+1}} x f( x, \mathbf{y}) d P_{X}(x) \otimes P _{ \mathbb{Y}} (\mathbf{y}) = \int_{\mathbb{R} ^{d+1}} x f( x, \mathbf{y}) dP _{ (X, \mathbb{Y})} (x,\mathbf{y}) \\
			&=& \mathbf{E} X f (X, \mathbb{Y}).
		\end{eqnarray*}
	\end{proof}\qed

	We also have a lemma in the converse direction. By $\Vert \cdot\Vert _{\infty}$ we denote the infinity norm on $\mathbb{R} ^{d+1}$.

	\begin{lemma}
		Consider a random vector $ (X, \mathbb{Y}) $ with $ \mathbf{E} \vert X\vert <\infty$. Assume that \begin{equation}\label{14n-1}
			\mathbf{E} \mathcal{N} f(X, \mathbb{Y})= 0
		\end{equation}
		for all differentiable functions $f: \mathbb{R} ^{d+1}\to \mathbb{R}$ with $ \Vert \partial _{x}f \Vert _{ \infty}<\infty.$ Then $ X \sim N (0, \sigma ^{2})$ and $ X$ is independent  of $\mathbb{Y}$.
	\end{lemma}
	\begin{proof}
		Let $\varphi $ be the characteristic function of the vector $ (X, \mathbb{Y})$, i.e.
		\begin{equation*}
			\varphi( \lambda_{1}, \boldsymbol{\lambda})= \mathbf{E} \left( e ^{i( \lambda _{1}X+ \boldsymbol{\lambda}\mathbb{Y}) }\right),
		\end{equation*}
		for $\lambda _{1} \in \mathbb{R}$ and $\boldsymbol{\lambda } \in \mathbb{R} ^{d}$. By applying (\ref{14n-1}) for the real and imaginary parts of $\varphi$, we get
		\begin{eqnarray*}
			&&	\partial _{\lambda _{1}} \varphi (\lambda_{1}, \boldsymbol{\lambda})= i \mathbf{E} \left( X e ^{i( \lambda _{1}X+\boldsymbol{ \lambda }\mathbb{Y}) }\right)\\
			&=&i \sigma ^{2} \mathbf{E} \left( \partial _{x} e ^{i( \lambda _{1}X+\boldsymbol{ \lambda} \mathbb{Y}) }\right) =-\lambda_{1} \sigma ^{2} \varphi( \lambda _{1},\boldsymbol{\lambda}).
		\end{eqnarray*}
		By noticing that for every $\boldsymbol{ \lambda} \in \mathbb{R}^{d}$, $\varphi (0, \boldsymbol{\lambda })= \varphi _{ \mathbb{Y}} (\boldsymbol{\lambda})$ (the characteristic function of the vector $\mathbb{Y}$), we obtain
		\begin{equation*}
			\varphi( \lambda _{1}, \boldsymbol{ \lambda} )=\varphi _{ \mathbb{Y}} (\boldsymbol{\lambda} )e ^{ -\frac{ \sigma ^{2} \lambda _{1} ^{2}}{2}},
		\end{equation*}
		and this implies $ X \sim N(0, \sigma ^{2}) $ and $ X$ independent of $\mathbb{Y}$. 
	\end{proof}\qed

	Let us now introduce  the multidimensional Stein's equation
	\begin{equation}
		\label{se}
		\mathcal{N}f (x, \mathbf{y}) = h(x, \mathbf{y}) - \mathbf{E} h ( Z, \mathbf{y}), \hskip0.3cm x \in \mathbb{R}, \mathbf{y}\in \mathbb{R} ^{d}
	\end{equation}
	where $Z\sim N(0, \sigma ^{2})$. In (\ref{se}), $h: \mathbb{R}^{d+1} \to \mathbb{R}$ is given and   we assume that $h$ is continuously differentiable with bounded partial derivatives. Let us show that (\ref{se}) admits a solution with suitable properties.
	
	\begin{prop}\label{pp11}
		Let $h: \mathbb{R} ^{d+1}\to \mathbb{R}$ be continuously differentiable with bounded partial derivatives. Then (\ref{se}) admits a unique bounded solution which is given by
		\begin{equation}
			\label{sse}f_{h}(x, \mathbf{y})= -\frac{1}{\sigma ^{2}}\int_{0} ^{1} \frac{1}{2\sqrt{t(1-t)} }\mathbf{E} \left[ Z h \left( \sqrt{t}x+\sqrt{1-t}Z, \mathbf{y} \right) \right] dt.
		\end{equation}
		Moreover, we have the following bounds:
		\begin{enumerate}
			\item \begin{equation}
				\label{bb1}\Vert f_{h}\Vert _{\infty} \leq \Vert \partial_{x_{1}} h \Vert _{\infty}.
			\end{equation}
			\item \begin{equation}
				\label{bb2}
				\Vert \partial _{x} f_{h}\Vert _{\infty} \leq \frac{1}{\sigma} \sqrt{\frac{2}{\pi}}\Vert \partial _{x} h\Vert _{\infty}.
			\end{equation}
			\item For $j=1,...,d$, if $\mathbf{y}=(y_{1},...,y_{d})$,
			\begin{equation}
				\label{bb3}
				\Vert \partial _{y_{j}} f_{h} \Vert _{\infty} \leq \frac{1}{\sigma} \sqrt{\frac{\pi}{2}} \Vert \partial _{x_{j}} h\Vert _{\infty},
			\end{equation}
		\end{enumerate}
	\end{prop}
	\begin{proof} By using the dominated convergence theorem, we get, by taking the derivative with respect to $x$ in (\ref{sse}),
		\begin{equation}\label{14n-3}
			\partial_{x}f_{h} (x, \mathbf{y}) =-\frac{1}{\sigma ^{2}} \int_{0} ^{1}  \frac{1}{2\sqrt{1-t} }\mathbf{E}\left[ Z \partial _{x}  h \left( \sqrt{t}x+\sqrt{1-t}Z, \mathbf{y} \right) \right].
		\end{equation}
		Now, we apply the standard Stein identity to the function $g(z)= h \left( \sqrt{t} x+ \sqrt{ 1-t}z, \mathbf{y}\right) $ and we obtain
		\begin{eqnarray}
			&&	\mathbf{E}\left[ Z \partial _{x}  h \left( \sqrt{t}+\sqrt{1-t}Z, \mathbf{y} \right) \right]\nonumber \\
			&=&\mathbf{E} g'(Z) = \sigma ^{2} \sqrt{1-t} \mathbf{E}\left[ \partial_{x} h \left( \sqrt{t} x + \sqrt{1-t} Z, \mathbf{y}\right)\right].\label{14n-2}
		\end{eqnarray}
		By plugging (\ref{14n-2}) into (\ref{sse}), the function $ f_{h}$ can be written as
		\begin{equation}\label{14n-4}
			f_{h} (x, \mathbf{y})= -\int_{0}^{1} \frac{1}{2\sqrt{t}} \mathbf{E} \left[ \partial_{x} h \left( \sqrt{t} x + \sqrt{1-t} Z, \mathbf{y}\right)\right]dt.
		\end{equation}
		By (\ref{14n-3}) and (\ref{14n-4}), we can write
		\begin{eqnarray*}
			&&	\partial_{x}f_{h} (x, \mathbf{y} ) - x f_{h} (x, \mathbf{y})\\
			&=&\int_{0} ^{1} \mathbf{E} \left[ \left( -\frac{Z} {2\sqrt{1-t}}+ \frac{x}{2\sqrt{t}}\right) \partial _{x} h \left( \sqrt{t}x+ \sqrt{1-t} Z, \mathbf{y}\right) \right] \\
			&=& \mathbf{E} \int_{0}^{1} \frac{d}{dt} h \left( \sqrt{t}x+ \sqrt{1-t} Z, \mathbf{y}\right) dt = h ( x, \mathbf{y} )-\mathbf{E} h (Z, \mathbf{y}).
		\end{eqnarray*}
		Consequently, $f_{h}$ given by (\ref{sse}) is a solution to (\ref{se}).  To prove (\ref{bb1}), we use (\ref{14n-4}) to get
		\begin{equation*}
			\Vert f_{h}\Vert _{\infty} \leq \int_{0}^{1} \frac{1}{2\sqrt{t}}\Vert \partial _{x_{1}} h\Vert _{\infty}\leq \Vert \partial _{x} h\Vert _{\infty}
		\end{equation*}
		The bound (\ref{bb2}) follows from (\ref{14n-3}) since
		\begin{equation*}
			\Vert \partial _{x} f_{h} \Vert _{\infty} \leq \frac{ \mathbf{E} \vert Z\vert}{\sigma ^{2}}\Vert \partial_{x} h\Vert _{\infty} \leq \sigma ^{-1} \sqrt{\frac{2}{\pi}}\Vert \partial_{x} h\Vert _{\infty}.
		\end{equation*}
		To prove (\ref{bb3}), we differentiate with respect to $y_{j}, j=1,...,d$ in (\ref{sse}),
		\begin{equation*}
			\partial _{y_{j}}f_{h}(x,\mathbf{y})= -\frac{1}{\sigma ^{2}}\int_{0} ^{1} \frac{1}{2\sqrt{t(1-t)} }\mathbf{E} \left[ Z \partial _{y_{j}} h \left( \sqrt{t}x+\sqrt{1-t}Z, \mathbf{y} \right) \right] dt
		\end{equation*}
		and
		
		\begin{equation*}
			\Vert \partial_{y_{j}} f_{h}\Vert _{\infty} \leq \frac{ \mathbf{E} \vert Z\vert}{\sigma ^{2}}\Vert \partial_{y_{j}} h\Vert _{\infty} \int_{0} ^{1} \frac{1}{2\sqrt{t(1-t)}}dt =\frac{1}{\sigma} \sqrt{\frac{\pi}{2}} \Vert \partial_{y_{j}} h\Vert _{\infty}.
		\end{equation*}
		To finish the proof, we notice that for any other solution $g_{h}$ to (\ref{se}), one has $$\partial _{x} \left( e ^{-\frac{x ^{2}}{2\sigma ^{2}}} \left( f_{h}(x,\mathbf{y}) - g_{h}(x,\mathbf{y})\right) \right) =0$$ so $g_{h}(x, \mathbf{y})= f_{h}(x,\mathbf{y})+e ^{\frac{x ^{2}}{2\sigma ^{2}}}  c(\mathbf{y})$ so $g_{h}$ is bounded if and only if $c(\mathbf{y})=0.$ 
	\end{proof}\qed
	
	By Proposition \ref{pp11}, if $f_{h}$ is the solution (\ref{sse}) to the Stein's equation (\ref{se}), we have
	\begin{equation*}
		\sigma ^{2} \partial _{x} f_{h} (x, \mathbf{y})-x f_{h}(x, \mathbf{y}) =h(x, \mathbf{y})-\mathbf{E} h(Z, \mathbf{y})
	\end{equation*}
	for any $h$ differentiable with bounded  partial derivatives.  Let $X, \mathbb{Y}$ be random vectors  with $\mathbf{E}\vert X\vert <\infty$. Let us integrate with respect to $\theta:= P_{(X, \mathbb{Y})}$ in the above identity. We have
	\begin{equation*}
		\int_{\mathbb{R} ^{d+1}} h(x,\mathbf{y} )d\theta (x,\mathbf{y})= \mathbf{E} h(X,\mathbb{Y})
	\end{equation*}
	and
	\begin{eqnarray*}
		&&	\int_{\mathbb{R} ^{d+1}} \mathbf{E}h(Z, \mathbf{y})d\theta (x, \mathbf{y})\\
		&=& \int_{\mathbb{R} ^{d+1}} \left( \int_{\mathbb{R}} h(z, \mathbf{y}) dP_{Z} (z)\right) d\theta (x,\mathbf{y})\\
		&=&\int_{\mathbb{R} ^{d}} \left(  \int_{\mathbb{R}} h(z, \mathbf{y}) dP_{Z} (z)\right)dP_{\mathbb{Y} } (\mathbf{y}) \\
		&=&\int_{\mathbb{R} ^{d+1}} h(x, \mathbf{y})dP_{Z}\otimes P_{\mathbb{Y}}(\mathbf{y})=\int_{\mathbb{R} ^{d+1}} h(x,\mathbf{y})d\eta (x,\mathbf{y}),
	\end{eqnarray*}
	with
	$$\eta= P_{Z}\otimes P_{\mathbb{Y}}.$$
	Therefore
	\begin{eqnarray}
		&&	\sigma ^{2}\mathbf{E} \partial _{x} f_{h} (X, \mathbb{Y} )-\mathbf{E}X f_{h}(X,\mathbb{Y})= \mathbf{E} h(X, \mathbb{Y})-\mathbf{E} h(Z',\mathbb{Y})\nonumber
		\\&=& \int_{\mathbb{R} ^{d+1}} h(x, \mathbf{y})d\theta (x,\mathbf{y})- \int_{\mathbb{R} ^{d+1}} h(x, \mathbf{y})d\eta (x,\mathbf{y})\label{14n-6}
	\end{eqnarray}
	where $Z'$ has the same law as $Z\sim N(0, \sigma ^{2})$  and $Z'$ is  independent of $\mathbb{Y}$.

	\subsection{Stein method and Malliavin calculus}

	Let 
	$$\mathcal{A}= \{ h: \mathbb{R}^{n}\to \mathbb{R},  h \mbox{ is Lipschitz continuous with } \Vert h\Vert _{Lip}\leq 1 \}$$
	and let $F,G$ be two $n$-dimensional random vectors  such that $ h(F), h(G)\in L ^{1}(\Omega)$ for every $h\in \mathcal{A}$. Then the Wasserstein distance between the probability distributions of $F$ and $G$ is defined by 
	\begin{equation}
		\label{dw}
		d_{W}(P_{F}, P_{G})=\sup_{h\in \mathcal{A}} \left| \mathbf{E}h(F)- \mathbf{E} h(G)\right|.
	\end{equation}
	We denoted by $\Vert h\Vert_{Lip} $  the Lipschitz norm of $h$ given by 
	\begin{equation*}
		\Vert h\Vert _{Lip}= \sup_{x,y\in \mathbb{R} ^{n}, x\not=y} \frac{ \vert h(x)-h(y)\vert}{\Vert x-y\Vert_{\mathbb{R} ^{n}}},
	\end{equation*} 
	with $\Vert \cdot \Vert_{\mathbb{R} ^{n}}$ the Euclidean norm in $\mathbb{R} ^{n}$.  The operators $D, L, \delta$ below are defined with respect to an isonormal process $(W(h), h\in H)$, see the Appendix. By $\langle \cdot, \cdot\rangle$ we denote the scalar product in the Hilbert space $H$.

	We use the ideas of the Stein method for normal approximation (see \cite{NP-book}) to prove the following result. 
	
	\begin{theorem}\label{tt1}
		Let $X$ be a  centered random variable in $\mathbb{D} ^{1,2}$ and let $\mathbb{Y}=(Y_{1},...,Y_{d})$ be such that $ Y_{j}\in \mathbb{D} ^{1,2}$ for all $j=1,...,d$. Let $\theta= P_{ (X, \mathbb{Y})}$ and $\eta = P_{Z}\otimes P _{\mathbb{Y}}$, where $ Z\sim N(0,\sigma ^{2})$. Then
		\begin{equation}\label{ineq1}
			d_{W} (\theta, \eta) \leq C \left( \mathbf{E} \left| \sigma ^{2} -\langle D(-L) ^{-1} X, DX\rangle \right| + \sum_{j=1} ^{d} \mathbf{E} \left| \langle D(-L) ^{-1} X, DY_{j}\rangle \right| \right). 
		\end{equation}
	\end{theorem}
	\begin{proof}
		Let $h: \mathbb{R} ^{d+1}\to \mathbb{R}$  be continuously differentiable wth bounded derivatives and let $f_{h}$ be the corresponding solution to the Stein's equation (\ref{se}). By using the well-known formula $ X= \delta D (-L) ^{-1} X$ in (\ref{14n-6}), we obtain, by integrating by parts 
		\begin{eqnarray*}
			&&\int_{\mathbb{R} ^{d+1}} h(x, \mathbf{y})d\theta (x, \mathbf{y} )- \int_{\mathbb{R} ^{d+1}} h(x, \mathbf{y} )d\eta (x, \mathbf{y})\\
			&=& \sigma ^{2} \mathbf{E}\partial _{x} f_{h} (X,\mathbb{Y})-\mathbf{E} \delta D (-L) ^{-1}X f_{h} (X,\mathbb{Y})\\
			&=& \sigma ^{2} \mathbf{E} \partial _{x} f_{h} (X,\mathbb{Y})-\mathbf{E} \langle D(-L) ^{-1} X, Df_{h} (X,\mathbb{Y})\rangle\\
			&=& \mathbf{E} \partial _{x}  f_{h} (X, \mathbb{Y} )\left( \sigma ^{2} - \langle D(-L) ^{-1} X, DX \rangle \right) \\
			&&- \mathbf{E}\sum_{j=1} ^{d} \partial _{x_{j}} f_{h} (X, \mathbb{Y})\langle D(-L) ^{-1} X, DY_{j} \rangle.
		\end{eqnarray*}
		Hence, by using inequalities (\ref{bb2}) and (\ref{bb3}) in Proposition \ref{pp11}, 
		\begin{eqnarray}
			&&	\left| \int_{\mathbb{R} ^{d+1}} h(x, \mathbf{y} )d\theta (x, \mathbf{y} )- \int_{\mathbb{R} ^{d+1}} h(x, \mathbf{y} )d\eta (x, \mathbf{y} )\right| \nonumber \\
			&& \leq C \left( \mathbf{E} \left| \sigma ^{2} -\langle D(-L) ^{-1} X, DX\rangle \right| + \sum_{j=1} ^{d} \mathbf{E} \left| \langle D(-L) ^{-1} X, DY_{j}\rangle \right| \right).\label{14n-8}
		\end{eqnarray}
		To finish the proof, we borrow again an argument from \cite{Pi} (proof of Lemma 9 in this reference) to approximate a Lipschitz function by continuously differentiable functions with bounded derivatives. Indeed, if $ h \in \mathcal{A}$ and $\eps>0$, then consider 
		\begin{equation*}
			h_{\eps} (x,y_{1}...,y_{d})= \mathbf{E} h \left( x+\sqrt{\eps} N,y_{1}+ \sqrt{\eps}N_{1},..., y_{d}+\sqrt{\eps} N_{d} \right), 
		\end{equation*}
		where $N, N_{1},...,N_{d}$ are independent standard normal random variables. Then $h_{\eps}$ is differentiable  and it safisfies
		\begin{equation*}
			\Vert h_{\eps} -h\Vert _{\infty} \to _{\eps \to 0}0, \hskip0.3cm \Vert \partial _{x} h_{\eps} \Vert _{\infty} \leq \Vert h_{\eps}\Vert _{Lip} \leq \Vert h\Vert _{ Lip} \leq 1
		\end{equation*}
		and
		\begin{equation*}
			\max_{j=1,...,d}\Vert \partial _{y_{j}} h_{\eps} \Vert _{\infty} \leq \Vert h_{\eps}\Vert _{Lip} \leq \Vert h\Vert _{ Lip} \leq 1.
		\end{equation*}
		Therefore, by (\ref{14n-8}), 
		\begin{eqnarray*}
			&&	\left| \int_{\mathbb{R} ^{d+1}} h(x, \mathbf{y} )d\theta (x, \mathbf{y})- \int_{\mathbb{R} ^{d+1}} h(x, \mathbf{y})d\eta (x, \mathbf{y})\right| \\
			&\leq &  2\Vert h_{\eps}-h\Vert _{\infty} + 	\left| \int_{\mathbb{R} ^{d+1}} h_{\eps}(x, \mathbf{y})d\theta (x, \mathbf{y})- \int_{\mathbb{R} ^{d+1}} h_{\eps}(x, \mathbf{y})d\eta (x, \mathbf{y})\right| \\
			&\leq & 2\Vert h_{\eps}-h\Vert _{\infty} + C \left( \mathbf{E} \left| \sigma ^{2} -\langle D(-L) ^{-1} X, DX\rangle \right| + \sum_{j=1} ^{d} \mathbf{E} \left| \langle D(-L) ^{-1} X, DY_{j}\rangle \right| \right)
		\end{eqnarray*}
		and we conclude by letting $\eps \to 0$.  
	\end{proof}\qed
	
	The corollary below is used to deal with random vectors with components in Wiener chaos. 
	\begin{corollary}\label{cor11}
		With the notation from Theorem \ref{tt1}, if $ X,Y_{1}..., Y_{d}\in \mathbb{D} ^{1,4}$, then
		\begin{equation}\label{31m-1}
			d_{W} (\theta, \eta) \leq C \left[ \left( \mathbf{E} \left| \sigma ^{2} -\langle D(-L) ^{-1} X, DX\rangle \right| ^{2}\right) ^{\frac{1}{2}}+ \sum_{j=1} ^{d} \left(\mathbf{E} \left| \langle D(-L) ^{-1} X, DY_{j}\rangle \right|^{2} \right) ^{\frac{1}{2}} \right]. 
		\end{equation}
	\end{corollary}
	\begin{proof}The proof follows from Theorem \ref{tt1}, by using Cauchy-Schwarz's inequality in the right-hand side of (\ref{ineq1}) and by noticing that $ \langle D(-L) ^{-1} X, DY_{j}\rangle $ belongs to $L ^{2} (\Omega)$ when $X, Y_{j} \in \mathbb{D} ^{1,4}$, for $j=1,2,...,d$. \end{proof}\qed

	\begin{remark}
		As a particular case of relation (\ref{ineq1}) in Theorem \ref{tt1}, it follows that if $ X_{1} \sim N(0, \sigma ^{2})$ and    $\langle DX_{1}, DX_{2} \rangle =0$ almost surely, then $X_{1}$ is independent of $X_{2}$. In particular, this means that, if $X_{1}=I_{1}(h)$ and $X_{2}=\sum_{n\geq 0} I_{n}(g_{n})$ (with  $h\in H, g_{n}\in H ^{\odot n}$ for every $n\geq 1$), then $h\otimes _{1} g_{n}=0 $ almost everywhere on $H ^{\otimes n-1}$ implies the independence of $X_{1} $ and $ X_{2}$.  This is related to the independence criterion for multiple stochastic integrals in \cite{UZ}, which states that two random variables $I_{p}(f)$ and $I_{q}(q)$ (with $f\in H ^{\odot p}, g\in H ^{\odot q}$) are independent if and only if $ f\otimes _{1}g$ vanishes almost everywhere on $ H ^{\otimes p+q-2}$. 
	\end{remark}

	\section{Asymptotic independence on Wiener chaos}
	The variant of the Stein's method presented in Section \ref{sec2} lead to some  strong consequences when it is applied to sequences of multiple  stochastic integrals. Here we describe and discuss our main findings in the case of the Wiener chaos. The proofs will be detailed in the next section. 
	
	\subsection{Preliminary tools}\label{sec31}
	Let us start with some auxiliary results that will be used several times in the sequel. Recall that $H$ is a real and separable Hilbert space and $W=(W(h), h\in H)$ is an isonormal process on the probability space $\left( \Omega, \mathcal{G}, P\right)$, where $\mathcal{G}$ is the sigma-algebra generated by $W$.  The operators $D, L$ and the multiple stochastic integral $I_{p}, p\geq 1$ are all with respect to $W$. 
	
	This our first auxiliary result.  The contraction of two kernels has been defined in the appendix (see (\ref{contra})).
	\begin{lemma}\label{ll1}
		Let $ f_{1}, f_{3}\in H ^{\odot p}$ and $ f_{2}, f_{4}\in H ^{\odot q}$ with $p,q\geq 1$. Then, for every $r=0,..., p\wedge q$,
		\begin{equation*}
			\langle f_{1}\otimes _{r}f_{2}, f_{3}\otimes _{r} f_{4}\rangle  _{ H ^{\otimes p+q-2r}}= 	\langle f_{1}\otimes _{p-r}f_{3}, f_{2}\otimes _{q-r} f_{4} \rangle _{ H ^{\otimes 2r}}.
		\end{equation*}
	\end{lemma}
	\noindent {\bf Proof: } This is e.g. Lemma 4.4 in \cite{T2}. \qed
	
	The following well-known result allows to express the $L^{2}$-norm of $\langle D(-L) ^{-1}X, DY\rangle $ when $X$ and $Y$ are multiple stochastic integrals. 
	\begin{lemma}\label{ll2}
		Let $X= I_{p}(f)$ and $Y= I_{q}(g) $ with $p,q\geq 1$ and $f\in H ^{\odot p}, g\in H ^{\odot q}$. Then
		\begin{equation*}
			\mathbf{E}	\langle D(-L)^{-1}X, DY\rangle ^{2} _{H} = (\mathbf{E}(XY))^{2} 1_{p=q}+ \sum_{r=1}^{p\wedge q} c(r,p,q) \Vert f\widetilde{\otimes }_{r} g\Vert ^{2} _{H ^{\otimes p+q-2r}},
		\end{equation*}
		where $c(r,p,q) $ are strictly positive combinatorial contants for $r=1,..., (p\wedge q)-1$ and 
		$$
		c(p\wedge q, p,q)=\begin{cases}
			0, \mbox{ if } p=q\\
			>0, \mbox {if } p\not=q.
		\end{cases}$$
	\end{lemma}
	\noindent {\bf Proof: }See e.g. \cite{NP-book}, Lemma 6.2.1. \qed 
	
	We will also need the celebrated Fourth Moment Theorem proven in \cite{NuPe}. See also \cite{NOT} for point 4. below. 
	
	\begin{theorem} (\cite{NuPe} and  \cite{NP-book})\label{4mom}
		Fix an integer $n\geq 1$. Consider a sequence $(F_k=I_n(f_k), k \geq 1)$ of square integrable random variables in the $n$th Wiener chaos.
		Assume that
		\begin{equation}
			\lim _{k\rightarrow \infty} \mathbf{E}[F_k^2] = \lim _{k\rightarrow \infty}n!  \| f_k\| ^2_{H^{\odot n}} =1.
		\end{equation}
		Then, the following statements are equivalent.
		\begin{enumerate}
			\item \label{4momO1} The sequence of random variables $(F_k=I_n(f_k), k\geq 1)$ converges to the standard normal law in distribution as $k\rightarrow \infty$.
			\item \label{4momO2} $\lim _{k\rightarrow \infty}\mathbf{E}[F_k^4]=3$.
			\item \label{4momO3} $\lim _{k\rightarrow \infty}\| f_k\otimes _l f_k\| _{H^{\otimes 2(n-l)}} =0$ for $l=1,2,\dots ,n-1$.
			\item \label{4momO4} $\| DF_k\| _H^2$ converges to $n$ in $L^2(\Omega)$ as $k\rightarrow \infty$.
		\end{enumerate}
	\end{theorem}

	\subsection{Main result}
	
	In this paragraph, we state our main findings and we discuss some  consequences. The main result of this work states as follows.  The notation  $d_{W}$ below stands for the Wasserstein distance, see (\ref{dw}).
	
	\begin{theorem}\label{tt2}
		Let us consider the integer numbers $p\geq2$, $d\geq 1$. Let $(X_{k}, k\geq 1)$ be a sequence of random variables such that for every $k\geq 1$, $X_{k}= I_{p} (f_{k})$ with $f_{k} \in H ^{\odot p}$. Assume that 
		\begin{equation}
			\label{cc1}
			X_{k}\to ^{(d) }_{k\to \infty} Z \sim N(0, \sigma ^{2}).
		\end{equation}
		Let $(\mathbb{Y} _{k}, k\geq 1)= \left( (Y_{1,k},..., Y_{d,k}), k\geq 1\right) $ be a sequence of random vectors such that,  for every $j=1,...,d,$ the random variable $Y_{j,k}$ belongs to $\mathbb{D} ^{1,2}$, and it admits the chaos expansion
		\begin{equation*}
			Y_{j,k}= \sum_{n=0} ^{\infty} I_{n} (g_{n,k} ^{(j)}) \mbox{ with } g_{n,k} ^{(j)} \in H ^{\odot n}
		\end{equation*}
		and 
		\begin{equation}
			\label{cc4}
			\sup_{k\geq 1}	\sum_{n=M+1}^{\infty}  n!  \Vert g_{n,k} \Vert ^{2}  _{ H ^{\otimes n}} \to _{M\to \infty }0.
		\end{equation}
		Suppose that there exists a random vector $\mathbb{U}$ in $\mathbb{R} ^{d}$ such that 
		\begin{equation}
			\label{cc2}
			\mathbb{Y}_{k} \to _{k\to \infty} \mathbb{U} \mbox{ in } L ^{2}(\Omega).
		\end{equation}
		Then, if 
		\begin{equation}\label{cc3}
			\mathbf{E} X_{k} Y_{j,k} \to_{k\to \infty}0 \mbox { for every } j=1,...,d
		\end{equation}	
		we have 
		\begin{equation}\label{n2}
			(X_{k}, \mathbb{Y}_{k}) \to ^{(d) }_{k\to \infty} (Z', \mathbb{U}),
		\end{equation}
		where $ Z'\sim N(0, \sigma ^{2})$ and $Z'$ is independent by the random vector $\mathbb{U}$. Moreover, for every $k\geq 1$,
		\begin{eqnarray}
			\label{est}
			&&	d_{W} \left( P_{ (X_{k}, Y_{k})}, P_{Z'}\otimes P_{\mathbb{U}}\right) \\
			&&\leq C \left[  \mathbf{E} \left|  \sigma ^{2}- \langle D(-L)^{-1}X_{k}, DX_{k} \rangle  \right|  + \sum_{j=1} ^{d} \mathbf{E} \left| \langle D(-L) ^{-1}X_{k}, DY_{j,k} \rangle _{H} \right| \right] + d_{W} (\mathbb{Y}_{k}, \mathbb{U}).\nonumber
		\end{eqnarray}
	\end{theorem} 
	
	Let us make some comment around Theorem \ref{tt2}. 
	
	\begin{itemize}
		\item Condition (\ref{cc4}) is automatically verified when $X_{j,k}$ belongs to a finite sum of Wiener chaoses or when $Y_{j,k}= Y_{j}$ for every $k\geq 1$ (this is stated in Corollary \ref{cor1}). On the other hand, this case (when the components of $\mathbb{Y}_{k}$ are in a finite sum of Wiener chaoses) will be proven before the main result, as a step of the proof of Theorem \ref{tt2}. 
		
		\item In Proposition \ref{pp4}, we show that if the components of $ \mathbb{Y}_{k}$ belong to the sum of the first $q$ Wiener chaoses ($q\leq p$), then it is enough to assume, instead of (\ref{cc2}),  only the convergence in law of $(\mathbb{Y}_{k}, k\geq1)$ in order to obtain (\ref{n2}). 
		
		\item The assumption (\ref{cc4}) also appears in the paper \cite{HN}, in the context of the normal approximation of Wiener space (see also Theorem 6.3.1 in \cite{NP-book}). 
		
		\item The quantitative bound (\ref{est}) is a direct consequence of the results in Section \ref{sec2}. It will be actually used inside the proof of the main result (Theorem \ref{tt2}). 
		
		\item The uncorrelation condition (\ref{cc2}) is obviously crucial for the joint convergence of $	(X_{k}, \mathbb{Y}_{k}) $ in Theorem \ref{tt2}. Another interesting question is what happens if we assume, instead of (\ref{cc3}), that
		\begin{equation*}
			\mathbf{E} X_{k} Y_{j,k} \to_{k\to \infty} c_{j},
		\end{equation*}
		with $c_{j}\not=0 $ for $j=1,...,d$. Can we deduce the joint convergence of $	(X_{k}, \mathbb{Y}_{k}) $ to a random vector with marginals $Z$ and $\mathbb{U}$? In the case when $\mathbb{U}$ follows a Gaussian distribution, the answer is given by the main result in \cite{PeTu}.  In order to give a complete answer, we need to know how to characterize the law of the vector $(Z, \mathbb{U})$ when $Z\sim N(0,\sigma^{2})$ is not independent of $\mathbb{U}$ and the law of $\mathbb{U}$ is not Gaussian.

	\end{itemize}

	Let us state the following corollary of the above theorem.

	\begin{corollary}\label{cor1}
		Consider the sequence $(X_{k}, k\geq 1)$ as in Theorem \ref{tt2} and $\mathbb{Y}=(Y_{1},...,Y_{d})$ be a random vector in $\mathbb{R} ^{d}$. Assume that for every $j=1,...,d$,  $Y_{j} \in \mathbb{D} ^{1,2}$ .
		Also assume 
		\begin{equation}\label{cc6}
			\mathbf{E}X_{k}Y_{j}\to _{k\to \infty}0.
		\end{equation} 
		Then  
		\begin{equation}
			\label{17i-1}
			(X_{k}, \mathbb{Y})\to ^{(d)}(Z', \mathbb{Y})
		\end{equation}
		with $Z'\sim N(0, \sigma ^{2})$ independent of $\mathbb{Y}$  and for $k\geq 1$,  
		\begin{eqnarray}
			\label{est2}
			&&	d_{W} \left( P_{ (X_{k}, Y)}, P_{Z'}\otimes P_{\mathbb{Y}}\right) \\
			&&\leq C \left[  \mathbf{E} \left|  \sigma ^{2}- \langle D(-L)^{-1}X_{k}, DX_{k} \rangle  \right|  + \sum_{j=1} ^{d} \mathbf{E} \left| \langle D(-L) ^{-1}X_{k}, DY_{j} \rangle _{H} \right| \right].\nonumber
		\end{eqnarray}
	\end{corollary}
\begin{proof}It is an immediate consequence of Theorem \ref{tt2}, since (\ref{cc4}) is obviously satisfied.  \end{proof}
	
	\begin{remark}
		Corollary \ref{cor1} actually says that any sequence in the $p$th Wiener chaos with $p\geq 2$ is asymptotically independent of any (regular enough) $d$-dimensional random vector in $L ^{2}(\Omega, \mathcal{G}, P)$  (with components in $ \mathbb{D} ^{1,2}$)  if the  uncorrelation assumption (\ref{cc6}) is satisfied.  
		
		Let us give a possible explanation of this phenomenon. Since $(X_{k}, k\geq 1) $ satisfies (\ref{cc1}), it follows from Theorem \ref{4mom} that, for $r=1,...,p-1$,
		\begin{equation*}
			\Vert f_{k}\otimes _{r}f_{k} \Vert _{ H ^{\otimes 2p-2r}}\to_{k\to \infty}0.
		\end{equation*}
		Let $h\in H$. Then, by Lemma \ref{ll1} and Cauchy-Schwarz' inequality,
		\begin{eqnarray*}
			&&	\Vert f_{k} \otimes _{1} h\Vert _{ H ^{\otimes p-1}}=\langle f_{k}\otimes _{1} h, f_{k}  \otimes _{1}h\rangle_  { H ^{\otimes p-1}}\\
			&=& \langle f_{k} \otimes _{p-1}f_{k}, h\otimes h\rangle _{ H ^{\otimes 2}}\leq \Vert f_{k} \otimes _{p-1}f_{k}\Vert _{H ^{\otimes 2}} \Vert h\Vert _{H} ^{2}\to_{k\to \infty }0. 
		\end{eqnarray*}
		This intuitively means, taking into account the independence criterion of two multiple integrals proven in \cite{UZ}, that $ X_{k}= I_{p}(f_{k})$ and $W(h)= I_{1}(h)$ are asymptotically independent for any $h\in H$. Then $X_{k}$ is asymptotically independent by any functional of $W$ and by density by any random variable in $ L^{2} (\Omega, \mathcal{G}, P)$ (recall that $\mathcal{G}$ is the sigma-algebra generated by $W$). 
		
	\end{remark}

	\section{Proof of the main result}
	The proof of the main result will be done into several steps. We start with an (intriguing) technical lemma (Lemma \ref{llkey} below) which plays a crucial role in our proofs. Then we prove the result in the case when the components of $\mathbb{Y}_{k}$ belong each of them to a Wiener chaos of fixed order, we continue with the case when these components are in a finite sum of Wiener chaos and finally we conclude the proof of Theorem \ref{tt2}. Our arguments use intensively the auxiliary tools recalled in Section \ref{sec31}, the Lemma \ref{llkey} and the Stein-Malliavin bounds  (\ref{ineq1}),  (\ref{31m-1})  obtained in Section \ref{sec2}.

	\subsection{A key lemma }
	As mentioned, the below lemma is a central point in our  approach. 
	
	\begin{lemma} \label{llkey}Let $p\geq 2$ and $q\geq 1$ be two integer numbers.  Let $(X_{k}, k\geq 1)$ be that such for every $k\geq 1$, $X_{k}= I_{p} (f_{k})$ with $f_{k} \in H ^{\odot p}$.  Assume
		\begin{equation}\label{cc11}
			X_{k} \to ^{(d)}_{k\to \infty} Z\sim N(0, \sigma ^{2}).
		\end{equation}
		Then, for every $g\in H ^{\odot q}$, 
		\begin{equation*}
			\Vert f_{k} \otimes _{r} g\Vert _{ H ^{ p+q-2r}}\to_{k\to \infty}0 \mbox{ for every } \begin{cases} 
				r=1,...,p\wedge q \mbox{ if } p\not=q\\
				r=1,...,(p\wedge q)-1 \mbox{ if }p=q.
			\end{cases} 
		\end{equation*}
	\end{lemma}
\begin{proof} Without loss of generality, we can assume that $H= L^{2}(T, \mathcal{B}, \nu)$, where $\nu$ is a sigma-finite measure without atoms. 
	
	Let $p> q$. Then the conclusion follows easily from Lemma \ref{ll1} and point 3. in the Fourth Moment Theorem (Theorem \ref{4mom}). Indeed, for every $1\leq r\leq q<p$, 
	\begin{eqnarray}
		&&\Vert f_{k}\otimes _{r} g\Vert ^{2} _{H ^{\otimes p+q-2r}}= \langle  f_{k}\otimes _{r} g,  f_{k}\otimes _{r} g\rangle _{H ^{\otimes p+q-2r}}=\langle f_{k}\otimes_{ p-r}f_{k}, g\otimes _{q-r} g\rangle _{H ^{ \otimes 2r}}\nonumber\\
		&\leq & \Vert f_{k} \otimes _{p-r}f_{k}\Vert _{H ^{\otimes 2r}} \Vert g\otimes _{q-r} g \Vert _{H ^{\otimes 2r}}\label{10a-1}
	\end{eqnarray}
	and $\Vert f_{k} \otimes _{p-k}f_{k}\Vert _{H ^{2r}} \to_{k\to \infty}0 $ by Theorem \ref{4mom} since $1\leq p-r\leq p-1$.  We employ the same  argument holds when $p=q$ and $ 1\leq r\leq p-1.$
	
	\vskip0.2cm
	Assume now $p<q$. If $ 1\leq r\leq p-1$, then the above argument still holds, due to the inequality 
	\begin{equation*}
		\Vert f_{k}\otimes _{r} g\Vert ^{2} _{H ^{\otimes p+q-2r}}\leq \Vert f_{k} \otimes _{p-r}f_{k}\Vert _{H ^{\otimes 2r}} \Vert g\otimes _{q-r} g \Vert _{H ^{\otimes 2r}}
	\end{equation*}
	and of the fact that $1\leq p-r\leq p-1$.
	
	It remains to prove that, for $2\leq p<q$, 
	\begin{equation}
		\label{4a-1}
		\Vert f_{k} \otimes _{p} g\Vert  _{ L ^{2}(T ^{q-p})}\to_{k\to \infty }0. 
	\end{equation}
	To prove (\ref{4a-1}), we will proceed into two steps.
	
	\vskip0.3cm
	
	\noindent {\bf Step 1. } We show that for every $h_{1},..., h_{q} \in H=L ^{2}(T)$, we have 
	\begin{equation*}
		\Vert f_{k}\otimes _{p} \left( h_{1}\tilde{\otimes}.....\tilde{\otimes }h_{q}\right) \Vert   _{ L ^{2}(T ^{q-p})}\to_{k\to \infty }0. 
	\end{equation*}
	We have
	\begin{equation*}
		h_{1}\tilde{\otimes}.....\tilde{\otimes }h_{q}=\frac{1}{q!}\sum_{\sigma \in S_{q}}h_{\sigma(1)}\otimes...\otimes h_{\sigma (q)},
	\end{equation*}
	where $S_{q}$ is the set of permutations of $\{1,...,q\}$. Then, via the definition of the contraction (\ref{contra}), 
	\begin{eqnarray*}
		&&\left(  f_{k}\otimes _{p} \left( h_{1}\tilde{\otimes}.....\tilde{\otimes }h_{q}\right) \right) (t_{1},...,t_{	q-p}) \\
		&=& \frac{1}{q!}\sum_{\sigma \in S_{q}}\int_{T ^{p}} f_{k}(u_{1},...,u_{p}) \left( h_{\sigma(1)}\otimes...\otimes h_{\sigma (q)}\right) (u_{1},...,u_{p}, t_{1},...,t_{q-p}) du_{1}...du_{p}\\
		&=&\frac{1}{q!}\sum_{\sigma \in S_{q}} \left( h_{\sigma (p+1)}\otimes...\otimes h_{\sigma(q)}\right) (t_{1},...,t_{q-p}) \\
		&&\times \int_{T ^{p}}f_{k}(u_{1},...,u_{p})\left( h_{\sigma(1)}\otimes...\otimes h_{\sigma (p)}\right) (u_{1},...,u_{p})du_{1}...du_{p}
		\\
		&=& \frac{1}{q!}\sum_{\sigma \in S_{q}} \left( h_{\sigma (p+1)}\otimes...\otimes h_{\sigma(q)}\right) (t_{1},...,t_{q-p}) \\
		&&\times \int_{T ^{p-1}} \left( \int_{T} f_{k}(u_{1},...,u_{p}) h_{\sigma (1)}(u_{1}) du_{1}\right) \left( h_{\sigma(2)}\otimes...\otimes h_{\sigma (p)}\right) (u_{2},...,u_{p})du_{2}...du_{p}\\
		&=& \frac{1}{q!}\sum_{\sigma \in S_{q}} \left( h_{\sigma (p+1)}\otimes...\otimes h_{\sigma(q)}\right) (t_{1},...,t_{q-p}) \\
		&&\times  \int_{T ^{p-1}}(f_{k}\otimes _{1} h_{\sigma (1)})(u_{2},...,u_{p})\left( h_{\sigma(2)}\otimes...\otimes h_{\sigma (p)}\right) (u_{2},...,u_{p})du_{2}...du_{p}\\
		&=& 	\frac{1}{q!}\sum_{\sigma \in S_{q}} \left( h_{\sigma (p+1)}\otimes...\otimes h_{\sigma(q)}\right) (t_{1},...,t_{q-p}) \langle f_{k}\otimes _{1} h_{\sigma (1)}, h_{\sigma(2)}\otimes...\otimes h_{\sigma (p)}\rangle _{L ^{2}(T ^{p-1})}
	\end{eqnarray*}
	Therefore,
	\begin{eqnarray*}
		&&	\Vert f_{k}\otimes _{p} \left( h_{1}\tilde{\otimes}.....\tilde{\otimes }h_{q}\right) \Vert   _{ L ^{2}(T ^{q-p})}\\
		&\leq & 	\frac{1}{q!}\sum_{\sigma \in S_{q}} \Vert h_{\sigma(p+1)}\Vert _{L ^{2}(T)}....\Vert h_{\sigma(q)}\Vert _{L ^{2}(T)}\left|  \langle f_{k}\otimes _{1} h_{\sigma (1)}, h_{\sigma(2)}\otimes...\otimes h_{\sigma (p)}\rangle _{L ^{2}(T ^{p-1})}\right| 
		\\
		&\leq & 		\frac{1}{q!}\sum_{\sigma \in S_{q}} \Vert h_{\sigma(p+1)}\Vert _{L ^{2}(T)}....\Vert h_{\sigma(q)}\Vert _{L ^{2}(T)} \Vert f_{k}\otimes _{1} h_{\sigma (1)}\Vert _{L ^{2}(T ^{p-1})} \Vert h_{\sigma(2)}\otimes...\otimes h_{\sigma (p)}\Vert  _{L ^{2}(T ^{p-1})}\\
		&\leq & 	\frac{1}{q!}\sum_{\sigma \in S_{q}} \Vert h_{\sigma(1)}\Vert _{L ^{2}(T)}.... \Vert h_{\sigma(q)}\Vert _{L ^{2}(T)} \sqrt{ \Vert f_{k}\otimes _{p-1} \otimes f_{k}\Vert _{ L ^{2} (T ^{2})}},
	\end{eqnarray*}
	where we used Lemma \ref{ll1} and Cauchy-Schwarz's inequality. We obtained
	\begin{equation*}
		\Vert f_{k}\otimes _{p} \left( h_{1}\tilde{\otimes}.....\tilde{\otimes }h_{q}\right) \Vert   _{ L ^{2}(T ^{q-p})}\leq \left( \prod_{j=1}^{q} \Vert h_{j}\Vert _{L ^{2}(T)}\right)  \sqrt{ \Vert f_{k}\otimes _{p-1} \otimes f_{k}\Vert _{ L ^{2} (T ^{2})}},
	\end{equation*}
	and this goes to zero as $k\to \infty$ by point 3. in Theorem \ref{4mom}. 
	
	\vskip0.3cm
	
	\noindent {\bf Step 2. } We prove the claim (\ref{4a-1}) for $g\in L^{2}_{S}(T^{q})$ (the set of symmetric functions in $ L^{2}(T ^{q})$).  Consider the sequence $ (g^{M}, M\geq 1)$ given by 
	\begin{equation*}
		g^{M}=\sum_{j_{1},...,j_{q}=1}^{M}\langle g, h_{j_{1}}\otimes....h_{j_{q}}\rangle _{ L^{2}(T ^{q})}h_{j_{1}}\otimes....\otimes h_{j_{q}}=\sum_{j_{1},...,j_{q}=1}^{M}\langle g, h_{j_{1}}\otimes....h_{j_{q}}\rangle _{ L^{2}(T ^{q})}h_{j_{1}}\widetilde{\otimes}....\widetilde{\otimes} h_{j_{q}}
	\end{equation*}
	where $(h_{i}, i\geq 1) $ is an orthonormal basis of $H=L^{2}(T)$. Then $g^{M}$ are symmetric functions and $\Vert g^{M}-g\Vert _{L ^{2}(T^{q})} \to_{M\to \infty}0$.  We write 
	\begin{equation*}
		f_{k}\otimes _{p}g= f_{k}\otimes _{p} g ^{M} + f_{k}\otimes _{p}g-f_{k}\otimes _{p}g^{M}
	\end{equation*}
	and
	\begin{equation}\label{4a-2}
		\Vert 	f_{k}\otimes _{p}g\Vert _{L ^{2}(T ^{q-p})}\leq \Vert f_{k}\otimes _{p} g ^{M}  \Vert _{L ^{2}(T ^{q-p})}+\Vert f_{k}\otimes _{p}g-f_{k}\otimes _{p}g^{M}\Vert _{L ^{2}(T ^{q-p})}.
	\end{equation}
	Now, for every $M\geq 1$,
	\begin{eqnarray}
		&&	\Vert f_{k}\otimes _{p}g-f_{k}\otimes _{p}g^{M}\Vert _{L ^{2}(T ^{q-p})}=\Vert f_{k}\otimes _{p} (g-g^{M})\Vert  _{L ^{2}(T ^{q-p})}\nonumber \\
		&\leq & \Vert f_{k}\Vert  _{L ^{2}(T ^{p})}\Vert g-g ^{M}\Vert  _{L ^{2}(T ^{q})}\leq C \Vert g-g ^{M}\Vert  _{L ^{2}(T ^{q})}. \label{4a-3}
	\end{eqnarray}
	We used the fact that, by (\ref{cc11}), $q!\Vert f_{k}\Vert ^{2} _{L ^{2}(T ^{p})}\to_{k\to \infty} \sigma ^{2}$ so the sequence $(f_{k}, k\geq 1)$  is  bounded in $ L^{2}(T ^{p})$. 
	
	Let $\eps>0$. By (\ref{4a-3}),  there exists $M_{0}\geq 1$ such that for any $M\geq M_{0} $
	
	\begin{eqnarray}\label{4a-4}
		\Vert f_{k}\otimes _{p}g-f_{k}\otimes _{p}g^{M}\Vert _{L ^{2}(T ^{q-p})}\leq \frac{\eps}{2}.
	\end{eqnarray}
	Take $M\geq M_{0}$. Then 
	\begin{equation*}
		f_{k}\otimes _{p} g ^{M}=\sum_{j_{1},...,j_{q}=1}^{M}\langle g, h_{j_{1}}\otimes....h_{j_{q}}\rangle _{ L^{2}(T ^{q})}\left( f_{k}\otimes _{p} (h_{j_{1}}\widetilde{\otimes}....\widetilde{\otimes} h_{j_{q}})\right).
	\end{equation*}By Step 1, 
	\begin{eqnarray*}
		\Vert f_{k}\otimes _{p} g ^{M}  \Vert _{L ^{2}(T ^{q-p})}\to_{k\to \infty} 0,
	\end{eqnarray*}
	so for $k$ large enough,
	\begin{equation}\label{4a-5}
		\Vert f_{k}\otimes _{p} g ^{M}  \Vert _{L ^{2}(T ^{q-p})}\leq \frac{\eps}{2}. 
	\end{equation}
	By plugging (\ref{4a-4}) and (\ref{4a-5}) into (\ref{4a-2}), we get the claim (\ref{4a-1}). 
	\end{proof}
	
	We state an immediate consequence of Lemma \ref{llkey}

		\begin{lemma} \label{llkey1}Let $p\geq 2$ and $q\geq 1$ be two integer numbers.  Let $(X_{k}, k\geq 1)$ be that such for every $k\geq 1$, $X_{k}= I_{p} (f_{k})$ with $f_{k} \in H ^{\odot p}$.  Assume (\ref{cc1}). 
	
\begin{enumerate}
	\item 
	Let $(g_{k}, k\geq 1)$ (with $g_{k}\in H ^{\odot q}$ for every $k\geq 1$) be a sequence that converges in $ H ^{\odot q}$. Then 
		\begin{equation}\label{n1}
			\Vert f_{k} \otimes _{r} g_{k}\Vert _{ H ^{ p+q-2r}}\to_{k\to \infty}0 \mbox{ for every } \begin{cases} 
				r=1,...,p\wedge q \mbox{ if } p\not=q\\
				r=1,...,(p\wedge q)-1 \mbox{ if }p=q.
			\end{cases} 
		\end{equation}
	\item 	Let $(g_{k}, k\geq 1)$ (with $g_{k}\in H ^{\odot q}$ for every $k\geq 1$) be a sequence bounded in $ H ^{\otimes q}$.  Assume $ q\leq p$.  Then (\ref{n1})  holds true.
\end{enumerate}
	\end{lemma}
\begin{proof}Denote by $g$ the limit in $ H ^{\otimes q}$ of the sequence $(g_{k}, k\geq 1)$. Then, for $r=1,...,p\wedge q $ (if $p\not=q$) or $r=1,..., p\wedge q -1$ (when $p=q$), we have
\begin{eqnarray*}
	\Vert f_{k} \otimes _{r} g_{k} \Vert _{H ^{\otimes p+q-2r}}&\leq & 	\Vert f_{k} \otimes _{r} g \Vert _{H ^{\otimes p+q-2r}}+ 	\Vert f_{k} \otimes _{r} (g_{k}-g) \Vert _{H ^{\otimes p+q-2r}}\\
	&\leq & 	\Vert f_{k} \otimes _{r} g \Vert _{H ^{\otimes p+q-2r}}+ \Vert f_{k}\Vert _{ H ^{\otimes p}}\Vert g_{k}-g\Vert _{H ^{\otimes q}}\\
	&\leq & 	\Vert f_{k} \otimes _{r} g \Vert _{H ^{\otimes p+q-2r}}+C\Vert g_{k}-g\Vert _{H ^{\otimes q}},
\end{eqnarray*}
since $(f_{k}, k\geq 1)$ is bounded in $ H ^{\otimes p}$. It suffices to apply Lemma \ref{llkey} to conclude point 1. Point 2. follows immediately from the bound (\ref{10a-1}) since
\begin{eqnarray*}
	&&\Vert f_{k}\otimes _{r} g_{k}\Vert ^{2} _{H ^{\otimes p+q-2r}}
		\leq  \Vert f_{k} \otimes _{p-r}f_{k}\Vert _{H ^{\otimes 2r}} \Vert g_{k}\otimes _{q-r} g_{k} \Vert _{H ^{\otimes 2r}}\\
		&&\leq  \Vert f_{k} \otimes _{p-r}f_{k}\Vert _{H ^{\otimes 2r}} \Vert g_{k}\Vert _{ H ^{\otimes q}}^{2} \leq C  \Vert f_{k} \otimes _{p-r}f_{k}\Vert _{H ^{\otimes 2r}}.
\end{eqnarray*}. \end{proof}

	\subsection{The proof of the main result when the components of $\mathbb{Y}_{k}$ belongs to a Wiener chaos}
	
	Let us make a first step to prove the main result, by dealing with the case when the random vector $\mathbb{Y}_{k}$ from the statement of Theorem \ref{tt2} has components that belong each of them in a Wiener chaos of fixed (but possibly different) order. 
	
	\begin{prop}
		\label{pp1}
		Let $p\geq 2$ and let $q_{1},...,.q_{d}\geq 1$ be integer numbers. Assume that $ (X_{k}, k\geq 1)$ is such that $ X_{k}= I_{p} (f_{k}), f_{k}\in H ^{\odot p}$ and (\ref{cc1}) holds true. Let $(\mathbb{Y} _{k}, k\geq 1)= \left( (Y_{1,k},..., Y_{d,k}), k\geq 1\right) $ be a sequence of random vectors such that
		for every $k\geq 1, j=1,...,d$, 
		\begin{equation*}
			Y_{j,k}= I_{q_{j}} (g_{j,k}) \mbox{ with } g_{j,k}\in H ^{\odot q_{j}}.
		\end{equation*}
		Suppose  (\ref{cc2}) and (\ref{cc3}). Then 
		\begin{equation}\label{6a-1}
			(X_{k}, \mathbb{Y}_{k}) \to ^{(d) }_{k\to \infty} (Z', \mathbb{U}),
		\end{equation}
		where $ Z'\sim N(0, \sigma ^{2})$ and $Z'$ is independent by $\mathbb{U}$. Moreover, we have the estimate (\ref{est}).
	\end{prop}

\begin{proof}  We first notice that (\ref{est}) is a direct consequence of (\ref{31m-1}) of the triangle's inequality.  Indeed, for every $k\geq 1$,
	\begin{eqnarray}
		d_{W}\left( P_{(X_{k}, \mathbb{Y}_{k})}, P_{Z'}\otimes P_{\mathbb{U}}\right)&\leq& 	d_{W}\left( P_{(X_{k}, \mathbb{Y}_{k})},  P_{Z'}\otimes P_{ \mathbb{Y}_{k}}\right)+d_{W} \left(  P_{Z'}\otimes P_{ \mathbb{Y}_{k}}, P_{Z'}\otimes P_{\mathbb{U}}\right)\nonumber\\
		&\leq & d_{W}\left( P_{(X_{k}, \mathbb{Y}_{k})},  P_{Z'}\otimes P_{ \mathbb{Y}_{k}}\right)+ d_{W} \left( P_{\mathbb{Y}_{k}}, P_{\mathbb{U}}\right)\nonumber \\
		&\leq & d_{W}\left( P_{(X_{k}, \mathbb{Y}_{k})},  P_{Z'}\otimes P_{ \mathbb{Y}_{k}}\right)+ \mathbf{E} \Vert \mathbb{Y}_{k}-\mathbb{U}\Vert _{ \mathbb{R} ^{d}}.\label{n4}
	\end{eqnarray}
	and then we use (\ref{31m-1}). For the rest of the proof, we will again proceed into several steps.
	
	\vskip0.2cm 
	
	\noindent {\bf Step 1. }We prove that for every $j=1,...,d$,
	\begin{equation}\label{30m-4}
		\mathbf{E} \langle D(-L) ^{-1} X_{k}, DY_{j,k} \rangle ^{2}\to _{k\to \infty}0. 
	\end{equation}
	By Lemma \ref{ll2}, we have, for every $k\geq 1$ and $j=1,...,d$, 
	\begin{equation}\label{30m-1}
		\mathbf{E} \langle D(-L) ^{-1} X_{k}, DY_{j,k} \rangle ^{2}=(\mathbf{E}X_{k} Y_{j,k})^{2} 1_{p=q_{j}} + \sum_{r=1}^{p\wedge q_{j}}  c(r,p,q_{j}) \Vert f_{k}\widetilde{\otimes}_{r} g_{j,k} \Vert ^{2} _{ H ^{\otimes p+q_{j}-2r}},
	\end{equation}
	where $c(r,p,q_{j})$ are as in Lemma \ref{ll2}. In particular, recall that $c(p\wedge q_{j}, p, q_{j})=0$ if $p\not=q_{j}$.  By Lemma \ref{llkey1},
	\begin{equation}\label{6a-7}
		\Vert f_{k}\widetilde{\otimes}_{r} g_{j,k} \Vert ^{2} _{ H ^{\otimes p+q_{j}-2r}}\leq 	\Vert f_{k}\otimes_{r} g_{j,k} \Vert ^{2} _{ H ^{\otimes p+q_{j}-2r}}\to_{k\to \infty}0
	\end{equation}
	for every $r=1,..., p\wedge q_{j}$ (if $ p\not=q_{j}$)  and $r=1,...,(p\wedge q_{j})-1$ (if $p=q_{j}$).  The relation (\ref{6a-7}) and the assumption (\ref{cc3}) imply the conclusion (\ref{30m-4}) of this step.

	\vskip0.2cm 
	
	\noindent {\bf Step 2. }Let us use the notation 
	\begin{equation}
		\label{not}
		\theta_{k}= P_{ (X_{k}, \mathbb{Y}_{k})}, \hskip0.2cm \eta_{k}= P_{ Z}\otimes P_{ \mathbb{Y}_{k}}, \hskip0.2cm \eta =P_{ Z}\otimes P_{\mathbb{U}}.
	\end{equation}
	In this step, we prove that
	\begin{equation}\label{30m-5}
		d_{W} (\theta_{k}, \eta_{k}) \to _{k\to \infty}0. 
	\end{equation}
	We know from (\ref{31m-1}) that 
	\begin{equation}\label{30m-8}
		d_{W}(\theta_{k}, \eta _{k}) \leq C \left[ \left( \mathbf{E} \left( \langle D(-L) ^{-1} X_{k}, DX_{k}\rangle -\sigma ^{2} \right) ^{2} \right) ^ {\frac{1}{2}} 
		+\sum_{j=1}^{d}  \left( \mathbf{E} \langle D(-L) ^{-1} X_{k}, DY_{j,k} \rangle ^{2} \right) ^{\frac{1}{2}}\right] 
	\end{equation}
	The assumption  (\ref{cc1}) and the Fourth Moment Theorem implies that (see Section 5 in \cite{NP-book}),
	\begin{equation*}
		\mathbf{E} \left( \langle D(-L) ^{-1} X_{k}, DX_{k}\rangle -\sigma ^{2}\right) ^{2} \to_{k\to \infty}0.
	\end{equation*}
	This fact, together with  Step 1, implies (\ref{30m-5}).

The conclusion is obtained by Step 2 and the bound (\ref{n4}).
 \end{proof}

	\begin{remark}\label{rem3}
		It is possible to write a quantitative bound for $d_{W}(\theta_{k}, \eta)$ in terms of the norms of the contractions of the kernels $f_{k}$ and $g_{j, k}$ (with the notation from Proposition \ref{pp1}). Indeed, assume $d=1$ and $q_{1}=q$. Then, by using (\ref{est}), Lemma \ref{ll2}, the inequality (\ref{10a-1}) and the fact that a sequence of random variables that converges in distribution in bounded in $ L ^{r}(\Omega)$ for every $r>1$ (see \cite{Ja} or \cite{NP-book}),  we can write 
		\begin{eqnarray*}
			d_{W}(\theta_{k}, \eta)&\leq& C\left[ (\mathbf{E}X_{k}Y_{k})^{2} 1_{p=q} + \sum_{r=1} ^{p-1} \Vert f_{k}\otimes _{r} f_{k}\Vert ^{2}_{ H ^{\otimes 2p-2r}}+ \sum_{r=1}^{ (p\wedge q)-1} \Vert f_{k}\otimes _{p-r}f_{k}\Vert _{ H ^{\otimes 2r}} \right. \\
			&&\left. + \Vert f_{k}\otimes _{q}  f_{k}\Vert _{ H ^{\otimes p-q}}1_{p>q} + \Vert f_{k}\otimes _{p}g_{k} \Vert ^{2}_{ H ^{ q-p}}1_{p<q}.\right] ^{\frac{1}{2}}.
		\end{eqnarray*}
		Taking into account point 3. in Theorem \ref{4mom},we can also write, for $k$ large enough, 
		\begin{eqnarray}\label{10a-2}
			d_{W}(\theta_{k}, \eta)\leq C\left[  \langle f_{k}, g_{k} \rangle ^{2}_{ H ^{\otimes p}}1_{p=q} + \sum_{r=1} ^{p-1}  \Vert f_{k}\otimes _{r} f_{k}\Vert _{ H ^{\otimes 2p-2r}}+\Vert f_{k}\otimes_{p} g_{k} \Vert ^{2}_{ H ^{ q-p}}1_{p<q}\right] ^{\frac{1}{2}}.
		\end{eqnarray}
		The above bound may be not optimal in some cases (see Remark \ref{rem4} in Section \ref{sec5}).
	\end{remark}

	\subsection{The components of $\mathbb{Y}_{k}$ belong to a finite sum of Wiener chaoses}
	Let us first notice that if a sequence of random variables $ (Y_{k}, k\geq 1)$ converges in $ L^{2} (\Omega)$ as $k\to \infty$ and 
	\begin{equation*}
		Y_{k}= \sum _{n=0} ^{\infty} I_{n} (g_{n,k}),  \hskip0.3cm g_{n,k}\in H ^{\odot n},
	\end{equation*}
then for every $n\geq 1$, the sequence $(g_{n,k}, k\geq 1)$ converges in $ H ^{\otimes n}$. 

We make a further step to get the main result by extending the result in Proposition \ref{pp1}.
	
	\begin{prop}\label{pp2}
		Assume that the sequence $(X_{k}, k\geq 1)$ is as in Proposition \ref{pp1} and let $\mathbb{Y}_{k}= (Y_{1,k},...,Y_{d,k})$  be such that for every $j=1,...,d$ and for every $k\geq 1$,
		\begin{equation*}
			Y_{j,k}=\sum_{n=0} ^{N_{0}} I_{n}(g_{n,k}^{(j)}) ,
		\end{equation*}
		with $N_{0}\geq 1$, $g_{n,k}^{(j)}\in H ^{\odot n}$ for $n\geq 0, k\geq 1$ and $=1,...,d$. Assume (\ref{cc2}) and (\ref{cc3}). Then
		\begin{equation}\label{5a-5}
			(X_{k}, \mathbb{Y}_{k})\to ^{(d)} _{k\to \infty} (Z', \mathbb{U})
		\end{equation}
		where $Z\sim N(0, \sigma ^{2}) $ and $Z',\mathbb{U}$ are independent. Moreover, the estimate (\ref{est}) holds true.
	\end{prop}
	\begin{proof} Recall the notation (\ref{not}).  Again,  the Stein-Malliavin bound (\ref{est}) follows directly from (\ref{31m-1}). By using this estimate (\ref{31m-1}), 
		\begin{equation*}
			d_{W} (\theta_{k}, \eta_{k}) \leq C\left( \mathbf{E}\left| \sigma ^{2}- \langle D(-L)^{-1}X_{k}, DX_{k} \rangle _{H}\right| + \sum_{j=1} ^{d} \mathbf{E} \left| \langle D(-L) ^{-1}X_{k}, DY_{j,k} \rangle _{H} \right| \right).
		\end{equation*}
		We also have, for every $j=1,...,d$ and $k\geq 1$, 
		\begin{eqnarray*}
			&&\mathbf{E} \left| \langle D(-L) ^{-1}, X_{k}, DY_{j,k} \rangle _{H}\right| = \mathbf{E} \left| \langle D(-L) ^{-1} X_{k}, \sum_{n=0} ^{N_{0}}DI_{n}(g_{n,k}^{(j)})\rangle _{H}\right|   \\
			&\leq & \sum_{n=0} ^{N_{0}} \mathbf{E}\left| \langle D(-L) ^{-1} X_{k}, DI_{n}(g_{n,k}^{(j)})\rangle _{H}\right| \leq  \sum_{n=0} ^{N_{0}} \left( \mathbf{E}\left| \langle D(-L) ^{-1} X_{k}, DI_{n}(g_{n,k}^{(j)})\rangle _{H}\right| ^{2} \right) ^{\frac{1}{2}}
		\end{eqnarray*}
		We notice that (\ref{cc3}) and the isometry of multiple stochastic integrals (\ref{iso}) implies that
		\begin{equation}
			\label{6a-2}
			\mathbf{E} X_{k} I_{n}(g_{n,k}^{(j)})\to_{k\to \infty }0,
		\end{equation}
		for every $j=1,...,d$ and for every $n=0,...,N_{0}$.  We use  Lemma \ref{ll2} to express the quantity $\mathbf{E}\left| \langle D(-L) ^{-1} X_{k}, DI_{n}(g_{n,k}^{(j)})\rangle _{H}\right| ^{2} $, and then by using (\ref{6a-2}) and Lemma \ref{llkey1}, we deduce that 
		\begin{equation*}
			\mathbf{E}\left| \langle D(-L) ^{-1} X_{k}, DI_{n}(g_{n,k}^{(j)})\rangle _{H}\right| ^{2} \to_{k\to \infty}0,
		\end{equation*}
		for every $j=1,...,d$ and $n-0,1, ...N_{0}$. Thus 
		\begin{equation*}
			\mathbf{E} \left| \langle D(-L) ^{-1}, X_{k}, DY_{j,k} \rangle _{H}\right|  \to_{k\to \infty}0,
		\end{equation*}
		for every $j=1,...,d$ and this implies
		\begin{equation*}
			d_{W}(\theta_{k}, \eta_{k}) \to _{k\to \infty}0.
		\end{equation*}
		To deduce (\ref{5a-5}),  it suffices to apply (\ref{n4}) in  the proof of Proposition \ref{pp1} and to use the hypothesis (\ref{cc2}). \end{proof}\qed

	\subsection{Proof of the main result (Theorem \ref{tt2})}
	Let $\eps >0$.  For $M\geq 1$, let us define,
	\begin{equation*}
		Y_{j,k}^{M}= \sum_{n=0} ^{M}  I_{n} (g_{n,k}^{(j)}), \hskip0.4cm j=1,...,d,
	\end{equation*}
	and consider the random vector in $\mathbb{R} ^{d}$
	\begin{equation}\label{ym}
		\mathbb{Y}^{M}_{k}=(Y_{1,k} ^{M},...,Y _{d,k} ^{M}), \hskip0.4cm k\geq 1.
	\end{equation}
	Clearly, for every $k\geq 1$, 
	\begin{equation*}
		\mathbf{E} \Vert \mathbb{Y}_{k}^{M}- \mathbb{Y}_{k} \Vert ^{2} _{ \mathbb{R} ^{d}}\to_{M\to \infty}0.
	\end{equation*}
	Recall that by $\Vert \cdot\Vert _ {\mathbb{R}^{d}}$ and $\langle \cdot, \cdot\rangle  _{\mathbb{R}^{d}}$ we denote  the Euclidean norm and the Euclidean scalar product in $\mathbb{R} ^{d}$.  By (\ref{cc3}) and the orthogonality of multiple stochastic integrals of different orders (\ref{iso}), for every $j=1,...,d$ and for every $M\geq 1$,
	\begin{equation}
		\label{6a-3}
		\mathbf{E} X_{k} Y_{j,k}^{M} \to _{k\to \infty} 0.
	\end{equation}
	Now, for any $\lambda_{1} \in \mathbb{R}$ and $\boldsymbol{\lambda}\in \mathbb{R}^{d}$, 
	\begin{eqnarray}
		&&	\left| \mathbf{E}e ^{i\lambda _{1} X_{k}+i\langle \boldsymbol{\lambda},\mathbb{Y}_{k}\rangle _{\mathbb{R}^{d}}}- \mathbf{E}e ^{i\lambda _{1} Z'} \mathbf{E} e ^{i\langle \boldsymbol{\lambda},\mathbb{U}\rangle _{\mathbb{R}^{d}}}\right| \nonumber\\
		&\leq &\left|  \mathbf{E}e ^{i\lambda _{1} X_{k} +i\langle \boldsymbol{\lambda},\mathbb{Y}_{k} \rangle _{\mathbb{R}^{d}}}-\mathbf{E}e ^{i\lambda _{1} X_{k}+i\langle \boldsymbol{\lambda},\mathbb{Y}^{M}_{k}\rangle _{\mathbb{R}^{d}}}\right| + \left| \mathbf{E}e ^{i\lambda _{1} X_{k}+i\langle \boldsymbol{\lambda},\mathbb{Y}^{M}_{k}\rangle _{\mathbb{R}^{d}}}- \mathbf{E}e ^{i\lambda _{1} Z'} \mathbf{E}e^{i\langle \boldsymbol{\lambda},\mathbb{Y}_{k} ^{M}\rangle _{\mathbb{R}^{d}}}\right|\nonumber \\
		&&+\left| \mathbf{E}e ^{i\lambda _{1} Z'} \mathbf{E}e^{i\langle \boldsymbol{\lambda},\mathbb{Y}_{k} ^{M}\rangle _{\mathbb{R}^{d}}}- \mathbf{E}e ^{i\lambda _{1} Z'} \mathbf{E}e^{i\langle \boldsymbol{\lambda},\mathbb{U} \rangle _{\mathbb{R}^{d}}}\right| \nonumber\\
		&=&a_{M,k}+ b_{M,k}+c_{M,k}. \label{8a-3}
	\end{eqnarray}
	Let us estimate separately  the three summands from above.  
	
	\vskip0.2cm
	
	\noindent {\it Estimation of $a_{M,k}$. } By the mean value theorem, 
	\begin{eqnarray}
		a_{M,k}&\leq & \mathbf{E} \left|  e ^{  i\langle \boldsymbol{\lambda},\mathbb{Y}_{k} \rangle _{\mathbb{R}^{d}}}-e ^{  i\langle \boldsymbol{\lambda},\mathbb{Y}_{k}^{M}  \rangle _{\mathbb{R}^{d}}}\right| 
		\leq  \mathbf{E} \Vert \mathbb{Y}^{M}_{k}- \mathbb{Y}_{k} \Vert  _{ \mathbb{R} ^{d}}\nonumber\\
		&\leq& \sqrt{ \sum_{j=1} ^{d} \mathbf{E} \left( Y ^{M}_{j,k}- Y_{j,k} \right) ^{2} }= \sqrt{ \sum_{j=1}^{d} \sum_{ n=M+1}^{\infty}n! \Vert g_{n,k} ^{(j)} \Vert ^{2}_{H ^{\otimes n}}}\nonumber\\
		&\leq &  \sqrt{ \sum_{j=1}^{d} \sup_{k\geq 1}\sum_{ n=M+1}^{\infty}n! \Vert g_{n,k} ^{(j)} \Vert ^{2}_{H ^{\otimes n}}}\label{8a-1}
	\end{eqnarray}
	and the last quantity goes to zero as $M\to \infty$ due to (\ref{cc4}). So, for $M\geq M_{1}$ large, $a_{M,k} \leq \eps$. 
	
	\vskip0.2cm
	
	\noindent {\it Estimation of $b_{M,k}$. }  Basically, the convergence of this term follows from Proposition \ref{pp2}, since the components of $ \mathbb{Y}_{k} ^{M}$ belong to a finite sum of Wiener chaoses.  For $M\geq M_{1}$, we have
	\begin{eqnarray*}
		&&	d_{W} \left( P_{ (X_{k}, \mathbb{Y}_{k}^{M})}, P_{Z'} \otimes P_{ \mathbb{Y} _{k}^{M}}\right) \\
		&\leq & C \left[ \mathbf{E} \left| \sigma ^{2} - \langle D(-L) ^{-1}X_{k}, DX_{k}\rangle _{H}\right| + \sum_{j=1}^{d} \mathbf{E} \left| \sigma ^{2} - \langle D(-L) ^{-1}X_{k}, DY^{M}_{j,k}\rangle _{H}\right| \right]
	\end{eqnarray*}
	Using (\ref{6a-3}), as in Proposition \ref{pp2}, the both summands in the right-hand above converge to zero as $k\to \infty$. So, for $k$ large, $b_{M,k}\leq \eps. $
	\vskip0.2cm
	
	\noindent {\it Estimation of $c_{M,k}$. } First notice that
	
	\begin{equation*}
		c_{M,k}\leq 	\left|\mathbf{E} e ^{i\langle \boldsymbol{\lambda}, \mathbb{Y}_{k} ^{M} \rangle _{\mathbb{R} ^{d}}}-\mathbf{E}e ^{i\langle \boldsymbol{\lambda}, \mathbb{U} \rangle _{\mathbb{R} ^{d}}}\right| 
	\end{equation*}
	Let $\eps >0$. We show that for $M,k$ large enough, 
	\begin{equation}\label{5a-7}
		\left|\mathbf{E} e ^{i\langle \boldsymbol{\lambda}, \mathbb{Y}_{k} ^{M} \rangle _{\mathbb{R} ^{d}}}-\mathbf{E}e ^{i\langle \boldsymbol{\lambda}, \mathbb{U} \rangle _{\mathbb{R} ^{d}}}\right| \leq \eps. 
	\end{equation}
	We have
	\begin{eqnarray}
		\left| \mathbf{E}e ^{i\langle \boldsymbol{\lambda}, \mathbb{Y}_{k} ^{M} \rangle _{\mathbb{R} ^{d}}}-\mathbf{E}e ^{i\langle \boldsymbol{\lambda}, \mathbb{U} \rangle _{\mathbb{R} ^{d}}}\right| &\leq & \left| \mathbf{E}e ^{i\langle \boldsymbol{\lambda}, \mathbb{Y}_{k} ^{M} \rangle _{\mathbb{R} ^{d}}}-\mathbf{E}e ^{i\langle \boldsymbol{\lambda}, \mathbb{Y}_{k} \rangle _{\mathbb{R} ^{d}}}\right| +\left| \mathbf{E}e ^{i\langle \boldsymbol{\lambda}, \mathbb{Y}_{k} \rangle _{\mathbb{R} ^{d}}}-\mathbf{E} e ^{i\langle \boldsymbol{\lambda}, \mathbb{U} \rangle _{\mathbb{R} ^{d}}}\right|\nonumber\\
		&\leq & C\mathbf{E} \Vert \mathbb{Y} _{k} ^{M}- \mathbb{Y}_{k} \Vert _{\mathbb{R}^{d}}+ \left| \mathbf{E} e ^{i\langle \boldsymbol{\lambda}, \mathbb{Y}_{k} \rangle _{\mathbb{R} ^{d}}}-\mathbf{E}e ^{i\langle \boldsymbol{\lambda}, \mathbb{U} \rangle _{\mathbb{R} ^{d}}}\right|.\label{8a-2}
	\end{eqnarray}
	We use the estimate (\ref{8a-1})
	\begin{eqnarray*}
		&&\mathbf{E} \Vert \mathbb{Y} _{k} ^{M}- \mathbb{Y}_{k} \Vert _{\mathbb{R}^{d}}
		\leq   \sqrt{ \sum_{j=1}^{d} \sup_{k\geq 1}\sum_{ n=M+1}^{\infty}n! \Vert g_{n,k} ^{(j)} \Vert ^{2}_{H ^{\otimes n}}}
	\end{eqnarray*}
	and the last quantity goes to zero as $M\to \infty$ due to (\ref{cc4}). By using this inequality and (\ref{cc2}) in (\ref{8a-2}), we get (\ref{5a-7}).  Therefore, for $k,M$ large, $c_{M,k}\leq \eps$. 
	
	Consequently, the left-hand side of (\ref{8a-3}) goes to zero as $k\to \infty$. \qed 
	
	It is possible to assume only the convergence in law of the sequence $(\mathbb{Y}_{k}, k\geq 1)$ instead of (\ref{cc2}) if the components of $ \mathbb{Y} _{k}$ belongs to the sum of the first $q$ Wiener chaos with $q\leq p$.

	\begin{prop}\label{pp4}
			Let us consider the integer numbers $p\geq2$, $d\geq 1$. Let $(X_{k}, k\geq 1)$ be a sequence of random variables such that for every $k\geq 1$, $X_{k}= I_{p} (f_{k})$ with $f_{k} \in H ^{\odot p}$ that satisfies (\ref{cc1}). 
	
		Let $(\mathbb{Y} _{k}, k\geq 1)= \left( (Y_{1,k},..., Y_{d,k}), k\geq 1\right) $ be a sequence of random vectors such that,  for every $j=1,...,d,$ the random variable $Y_{j,k}$ belongs to $\mathbb{D} ^{1,2}$, and it admits the chaos expansion
		\begin{equation*}
			Y_{j,k}= \sum_{n=0} ^{q} I_{n} (g_{n,k} ^{(j)}) \mbox{ with } g_{n,k} ^{(j)} \in H ^{\odot n}
		\end{equation*}
	with $q\leq p$. Suppose that there exists a random vector $\mathbb{U}$ in $\mathbb{R} ^{d}$ such that 
		\begin{equation}
			\label{cc22}
			\mathbb{Y}_{k} \to ^{(d)}_{k\to \infty} \mathbb{U}.
		\end{equation}
		Then, if (\ref{cc3}) holds true, 	we have 
		\begin{equation*}
			(X_{k}, \mathbb{Y}_{k}) \to ^{(d) }_{k\to \infty} (Z', \mathbb{U}),
		\end{equation*}
		where $ Z'\sim N(0, \sigma ^{2})$ and $Z'$ is independent by the random vector $\mathbb{U}$. Moreover, (\ref{est}) holds true.
	\end{prop}
\begin{proof}
	The proof can be done by following the lines of the proof of Proposition \ref{pp2}, by using point 2. in Lemma \ref{llkey1}. We use the notation (\ref{not}). Via the bound (\ref{30m-8}) and point 2. in Lemma \ref{llkey1}, we  obtain that
	\begin{equation}\label{n3}
		d_{W}(\theta_{k}, \eta_{k}) \to _{k\to \infty} 0.
	\end{equation}
 Let $ f: \mathbb{R} ^{d+1} \to \mathbb{R}$ be a continuous and bounded function. By using the triangle's inequality, we have 
\begin{eqnarray*}
	&&\left| \int_{\mathbb{R} ^{d+1}} f(x)d\theta_{k} (x)-  \int_{\mathbb{R} ^{d+1}} f(x)d\eta (x)\right| \\
	&&\leq \left| \int_{\mathbb{R} ^{d+1}} f(x)d\theta_{k} (x)-  \int_{\mathbb{R} ^{d+1}} f(x)d\eta_{k}  (x)\right| + \left| \int_{\mathbb{R} ^{d+1}} f(x)d\eta_{k} (x)-  \int_{\mathbb{R} ^{d+1}} f(x)d\eta (x)\right|.
\end{eqnarray*}
The first summand in the right-hand side converges to zero as $k\to \infty$ by (\ref{n3}). The second summand in the right-hand side also goes to zero as $k$ tends to infinity due to the assumption (\ref{cc22}). Then, the conclusion  is obtained. 
	\end{proof}

\subsection{A counter-example}\label{counter}
Assume $(X_{k}= I_{p}(f_{k}), k\geq 1)$ with $ f_{k} \in H ^{\odot p}$ be such that $ X_{k}\to _{k\to \infty} Z\sim N(0, \sigma ^{2})$.  Let $(Y_{k}, k\geq 1)$ be a sequence in the $q$th Wiener chaos, $Y_{k}= I_{q}(g_{k}), g_{k} \in H ^{\odot q}$. Assume that $q>p$
\begin{equation*}
	Y_{k} \to _{k\to \infty} U.
\end{equation*}
	Can we deduce the joint convergence of $ (X_{k}, Y_{k}) $ to $(Z', U)$ where $ Z'\sim N(0, \sigma ^{2}) $ and $Z', U$ are independent? By Theorem \ref{tt2} and Proposition \ref{pp4}, the conclusion is true if the convergence of $ (Y_{k}, k\geq 1)$ holds in $ L ^{2}(\Omega)$ or if $p\geq q$ (and (\ref{cc3}) holds). For $q>p$, the answer is negative as illustrated by the following example. Let
	\begin{equation*}
		g_{k}= f_{k} \widetilde{\otimes} f_{k}, \hskip0.5cm k\geq 1,
	\end{equation*}
	and $Y_{k}= I_{2p} (g_{k}), k\geq 1$. Then, by the product formula (\ref{prod}),
	\begin{equation*}
		X_{k} ^{2} -\mathbf{E} X_{k}^{2}= Y_{k}+ R_{k}, 
	\end{equation*}
	where $ R_{k} \to _{k\to \infty} 0$ in $ L ^{2} (\Omega)$ (this comes from point 3. in Theorem \ref{4mom}). Consequently, 
	\begin{equation*}
		(X_{k}, Y_{k}) \to ^{(d)} _{k\to \infty} ( Z, Z ^{2}- \sigma ^{2}),
	\end{equation*}
and obviously the components of the limit vector are not independent. 
	
	\section{Applications}
	We illustrate our results by four examples. In the first example, we deduce from Proposition \ref{pp1} the joint convergence  of the Hermite variations of $d+1$ correlated fractional Brownian motions. The second example constitutes an application of Theorem \ref{tt2}, by considering a random variable with infinite chaos expansion. In the third  example, we treat a  two -dimensional sequence in Wiener chaos, one component being asymptotically Gaussian and the second component satisfying a non-central limit theorem. Such estimates are new in the literature and they cannot be obtained via the standard Stein method.  Finally, in the last example, to evaluate the dependence structure between  the solution a stochastic differential equation and the random noise.

	\subsection{Hermite variations of correlated fractional Brownian motions}
	Let $ (W_{t}, t\geq 0) $ be a Wiener process  and  for $ H\in (0,1), t\geq 0$ consider the kernel
	\begin{equation*}
		f_{t, H} (s)= d(H) \left( (t-s)_{+} ^{H-\frac{1}{2}} -(-s) _{+}^{H-\frac{1}{2}}\right), \hskip0.3cm s\in \mathbb{R}.
	\end{equation*}
	where $d(H)$ is a normalizing constsnt that ensures that $\int_{\mathbb{R} }f_{t,H}(s) ^{2}ds = t ^{2H}.$ Let $ H_{0},H_{1},..., H_{d} \in (0,1)$ and define, for $i=0,1, ...,d$, 
	\begin{equation}\label{22i-1}
		B ^{H_{i}} _{t} = \int_{\mathbb{R}} f_{t, H_{i}} (s) dW_{s}, \hskip0.4cm t\geq 0. 
	\end{equation}
	Then , for $i=0,1,...,d$, $ \left( B^{H_{i}}, t\geq 0\right) $ are $d+1$ (correlated) fractional Brownian motions with Hurst parameters $H_{i}$. We write, for any integer number  $k\geq 0$,
	\begin{equation*}
		B ^{H_{i}}_{k+1}- B ^{H_{i}} _{k} = I_{1} (L_{k, H_{i}}), \hskip0.3cm i=0,1,...,d,
	\end{equation*}
	where $ I_{q}$ stands for the multiple stochastic integral of order $q\geq 1$ with respect to the Wiener process $W$ and for $k\geq 0$,
	\begin{equation}
		L_{k, H_{i}}= f_{k+1, H_{i}}- f_{k, H_{i}}.
	\end{equation}
	
	For $N\geq 1$ integer, we set 
	\begin{equation}
		\label{xn}
		X_{N}= \frac{1}{\sqrt{N}}\sum_{k=0} ^{N-1} I_{p} \left( L_{k, H_{0}}^{\otimes p}\right) = I_{p} (f_{N})
	\end{equation}
	and for $j=1,...,d$,
	\begin{equation}\label{yn}
		Y_{N,j} = N ^{q_{j}(1-H_{j})-1}\sum_{k=0} ^{N-1} I_{q} \left( L_{k, H_{j}}^{\otimes q_{j}}\right)= I_{q_{j}}(g_{N,j}). 
	\end{equation}
	We used the notation
	\begin{equation}
		\label{fg}
		f_{N}= \frac{1}{\sqrt{N}}\sum_{k=0} ^{N-1} L_{k, H_{0}}^{\otimes p} \mbox{ and } g_{N,j}= N ^{q_{j}(1-H_{j})-1}\sum_{k=0} ^{N-1} L_{k, H_{j}}^{\otimes q_{j}}
	\end{equation}
	From the  classical Breuer -Major theorem  (see \cite{BM}) we know  the limit  behavior in distribution of the sequence (\ref{xn}) while the Non-Central limit theorem (see e.g. \cite{Ta2})  gives the limit behavior of (\ref{yn}). More precisely, we have the following. 
	
	\begin{theorem}\label{tt3}
		Consider the sequences $(X_{N}, N\geq 1)$ and $ (Y_{N,j}, N\geq 1)$ given by (\ref{xn}), (\ref{yn}), respectively. Then
		\begin{enumerate}
			\item If $H_{0} \in \left(0, 1-\frac{1}{p}\right)$,
			\begin{equation*}
				X_{N} \to ^{(d)} _{N\to \infty} N (0, \sigma _{p, H_{0}}^{2}).
			\end{equation*}
			\item If $ H_{j} \in \left( 1-\frac{1}{2q_{j}}, 1\right)$ for $j=1,...,d$,
			\begin{equation*}
				Y_{N,j} \to ^{(d)} _{N\to \infty} c_{ q_{j}, H_{j}}R^{\gamma_{j}} _{1} ,
			\end{equation*}
			where $ R^{\gamma_{j}}_{1}$ is a Hermite random variable with  Hurst parameter $\gamma_{j}=1+q(H-1)$.  The explicit expression of the constants $\sigma_{p, H_{0}}, c_{q_{j}, H_{j}}>0$ can be found in e.g. \cite{BM}, \cite{Ta2}. 
		\end{enumerate}
	\end{theorem} 	
	Recall that the Hermite random variable has a non-Gaussian law (it actually lives in $q$th Wiener chaos) and it represents the value at time $t=1$ of a Hermite process. For more details on Hermite processes, see e.g. \cite{T}.

	Let 
	\begin{equation*}
		\mathbb{Y}_{N}=(Y_{N,1},\ldots, Y_{N,d}), \hskip0.5cm N\geq 1.
	\end{equation*}The purpose is to show the joint convergence of the two-dimensional random sequence $((X_{N}, \mathbb{Y}_{N}), N\geq 1)$. 	Let us recall some facts. For every integers $k,l\geq 1$ and for $i,j=0,1,...,d$ (see \cite{MaTu2}),
	\begin{equation*}
		\mathbf{E} (B ^{H_{i}}_{k+1}- B ^{H_{i}}_{k})(B ^{H_{j}}_{l+1}- B ^{H_{j}}_{l})= \langle L_{k, H_{i}}, L_{l, H_{j}}\rangle _{ L ^{2}(\mathbb{R})}= D(H_{i}, H_{j}) \rho_{\frac{ H_{i}+ H_{j}}{2}}(k-l),
	\end{equation*}
	where $D(H_{i}, H_{j})$ is a constant depending on $ H_{i}, H_{j}$ and  for $v\in \mathbb{Z}$, 
	\begin{equation}\label{rho}
		\rho_{H}(v)= \frac{1}{2} \left( \vert v+1\vert  ^{2H}+\vert v-1\vert -2\vert v\vert ^{2H}\right). 
	\end{equation}
	For $v$ sufficiently large, one has 
	\begin{equation}\label{bro}
		\vert \rho_{H}(v) \vert \leq C_{H} v ^{2H-2}
	\end{equation}
	
	We have the following result. 
	
	\begin{prop}\label{pp3}
		Let $p\geq 1, q_{1},....,q_{d}\geq 2$  be  integer numbers such that $p\geq \max(q_{1},..., q_{d})$ and assume that  for $j=1,...,d$, 
		\begin{equation}\label{31m-3}
			0<	H_{1} < 1-\frac{1}{2p} \mbox{ and } 1-\frac{1}{2q_{j}}< H_{j} < 1.
		\end{equation}
		Consider the sequences $(X_{N}, N\geq 1)$ and $(\mathbb{Y}_{N}, N\geq 1)$ given by (\ref{xn}) and (\ref{yn}), respectively. Then 
		\begin{equation*}
			(X_{N}, \mathbb{Y} _{N}) \to ^{(d)} _{N\to \infty} (Z, c_{q_{j},H_{j}}R^{\gamma_{j}}_{1}, j=1,...,d), 
		\end{equation*} 
		where $Z\sim N (0, \sigma _{p, H_{0}}^{2})$ and $R^{\gamma_{j}}_{1} $ stands for a Hermite random variable (with Hurst index $\gamma_{j}$) independent of $Z$. The constants $\sigma _{p, H_{0}}$ and $ c_{q_{j},H_{j}}$ are those from Theorem \ref{tt3}.
	\end{prop}
	\begin{proof}First, we notice that, as $N\to \infty$,
		\begin{equation}
			\label{21i-1}
			\mathbb{Y}_{N}\to ^{(d)} (c_{q_{1}, H_{1}}R ^{\gamma_{1}}_{1},..., c_{q_{d}, H_{d}}R^{\gamma_{d}}_{1}).
		\end{equation}
		The above claim can be argued in the following way: for every $c>0$, we have the scaling property 
		\begin{equation*}
			\left( B ^{H_{1}}_{ct},..., B ^{H_{d}}_{ct}, t\geq 0\right)\equiv ^{(d)}\left( c ^{H_{1}} B ^{H_{1}}_{t},..., c ^{H_{d}} B ^{H_{d}}_{t}, t\geq 0\right),
		\end{equation*}
		where $"\equiv ^{(d)}" $ means the equivalence of finite dimensional distributions. This is a consequence of (\ref{22i-1}) and of the scaling property of the Wiener process $W$. Then, for all $N\geq 1$,  we have the equality in law
		\begin{equation*}
			(Y_{N,1},..., Y_{N,d})= ^{(d)} ( Y '_{N,1},..., Y '_{N,d})
		\end{equation*}
		where, for every $j=1,...,d$, 
		\begin{equation*}
			Y '_{N, j} =q_{j}! N ^{q_{j}-1} \sum_{k=0} ^{N-1} H_{q_{j}}\left( B ^{H_{j}}_{ \frac{k+1}{N}}-  B ^{H_{j}}_{ \frac{k}{N}}\right)
		\end{equation*}
		with $H_{q}$ the Hermite polynomial of degree $q$. On the other hand, for every $j=1,...,d$, the sequence $ (Y '_{N,j}, N\geq 1)$ converges in $ L^{2} (\Omega)$, as $N\to \infty$, to $ c_{q_{j}, H_{j}}R^{\gamma_{j}}_{1}$ (see e.g. \cite{NP-book}). This implies (\ref{21i-1}). 
		
		In order to apply Proposition \ref{pp1}, we just need to check (\ref{cc3}). Obviouly, this holds for $p\not=q_{j}$, since in this situation $ \mathbf{E} X_{N} Y_{N,j}=0$ for all $N\geq 1$ and for all $j=1,...,d$. We calculate $\mathbf{E}X_{N} Y_{N,j}$ for $p=q_{j}$.  We have, by the isometry formula (\ref{iso}),
		\begin{eqnarray*}
			\mathbf{E} X_{N} Y _{N,j} &=& p! N ^{ p(1-H_{j})-\frac{3}{2}}\sum_{k,l=0} ^{N-1} \langle L_{k, H_{0}}, L_{l, H_{j}}\rangle _{ L ^{2}(\mathbb{R})}^{p} \\
			&=& p! D(H_{0}, H_{j}) ^{p} N ^{ p(1-H_{j})-\frac{3}{2}}\sum_{k,l=0} ^{N-1}  \rho_{\frac{ H_{0}+ H_{j}}{2}}(k-l)^{p},
		\end{eqnarray*}
		and for $N$ large enough, by (\ref{bro}),
		\begin{eqnarray*}
			\vert \mathbf{E} X_{N} Y _{N,j} \vert &\leq & c(H_{0}, H_{j},p) N ^{ p(1-H_{j})-\frac{3}{2}}\left( 1+ \sum_{k=1}^{N} (N-k) k ^{ (H_{0}+ H_{j}-2)p}\right) \\
			&\leq & c(H_{0}, H_{j},p) N ^{ p(1-H_{j})-\frac{3}{2}}\left( 1+N\sum_{k=1}^{N}  k ^{ (H_{0}+ H_{j}-2)p}\right).
		\end{eqnarray*}
		
		Assume $ (H_{0}+ H_{j}-2)p<-1$. In this case, the series $\sum_{k\geq 1} k ^{ (H_{0}+ H_{j}-2)p}$ converges and we get 
		\begin{equation*}
			\vert 	\mathbf{E} X_{N} Y_{N,j} \vert \leq  c(H_{0}, H_{j},p) N ^{ p(1-H_{j})-\frac{1}{2}}\to_{N\to \infty }0
		\end{equation*}
		since $H_{j}>1-\frac{1}{2p}$.

		Assume  $ (H_{0}+ H_{j}-2)p>-1$. Then the sequence $\sum_{k=1}^{N}  k ^{ (H_{0}+ H_{j}-2)p}$ behaves as $ N ^{ (H_{0}+ H_{j}-2)p+1}$ for $N$ large and thus 
		\begin{eqnarray*}
			\vert \mathbf{E} X_{N} Y _{N,j}\vert &\leq &  c(H_{0}, H_{j},p) N ^{ p(1-H_{j})-\frac{3}{2}}\left( 1+ N ^{ (H_{0}+ H_{j}-2)p+1}\right)\\
			&=&  c(H_{0}, H_{j},p)\left( N ^{ p(1-H_{j})-\frac{3}{2}}+N ^{-p(1-H_{0})+ \frac{1}{2}}\right) \to _{N\to \infty}0,
		\end{eqnarray*}
		since $ H_{0}<1-\frac{1}{2p}$ and $H_{j}>1-\frac{1}{2p}$. 
		
		If $ (H_{0}+ H_{j}-2)p=-1$, then $\sum_{k=1}^{N}  k ^{ (H_{0}+ H_{j}-2)p}$ behaves as $\log (N)$ and 
		\begin{equation*}
			\vert \mathbf{E} X_{N} Y _{N,j}\vert \leq   c(H_{0}, H_{j},p) N ^{ p(1-H_{j})-\frac{1}{2}}\log(N)\to _{N\to \infty}0.
		\end{equation*}
		We obtained
		\begin{equation*}
			\vert \mathbf{E} X_{N} Y _{N,j}\vert \leq   c(H_{0}, H_{j},p)
			\begin{cases}
				N ^{ p(1-H_{j})-\frac{1}{2}} \mbox{ if } (H_{0}+H_{j}-2)p< -1\\
				N ^{ p(1-H_{j})-\frac{1}{2}}\log(N), \mbox{ if } (H_{0}+ H_{j}-2)p=-1\\
				N ^{ p(1-H_{j})-\frac{3}{2}}+N ^{-p(1-H_{0})+ \frac{1}{2}}\mbox{ if } (H_{0}+H_{j}-2)p>-1.
			\end{cases}
		\end{equation*}
		In particular $\mathbf{E}X_{N}Y_{N,j}\to _{N\to \infty}0$ and (\ref{cc3}) holds. The conclusion follows by Proposition \ref{pp1}. 
	\end{proof}\qed 
	
	\begin{remark}
		\begin{enumerate}
			
			\item A quantitative bound in Proposition \ref{pp3} can be obtained via (\ref{est}) or (\ref{10a-2}). 
			
			\item Let the above notation prevail.  It is also possible to apply Proposition \ref{pp3} to the estimation of the Hurst parameter $(H_{0},H_{1},..., H_{d})$ from the discrete observations $\left( B ^{H_{j}}_{\frac{i}{N}}, i=0,1,..., N, j=0,1,...,d\right)$. Denote, for $j=0,1,...,d$,
			\begin{equation*}
				S_{N,j}= \frac{1}{N} \sum_{i=0}^{N-1} \left( B ^{H_{j}}_{\frac{i+1}{N}}- B ^{H_{j}}_{\frac{i}{N}}\right) ^{2}.
			\end{equation*}
			Then  
			\begin{equation*}
				\widehat{H}_{N,j}=-\frac{ \log(S_{N,j})}{2\log(N)}, \hskip0.3cm j=0,1,...,d
			\end{equation*}
			are consitent estimators for the Hurst index $H_{j}$  and (see e.g. Section 5.5 in \cite{T})
			\begin{equation*}
				2\sqrt{N} (\widehat{H}_{N,0}- H_{0})= X_{N}+R_{N,0} 
			\end{equation*}
			and for $j=1,...,d$,
			\begin{equation*}
				2N ^{2-2H_{j}}(\widehat{H}_{N,j}- H_{j})=Y_{N,j}+ R_{N,j}
			\end{equation*}
			where $R_{N,j}, j=01,...,d$ converge almost surely to zero as $N\to \infty$. From Proposition \ref{pp3}, we get the joint convergence in law, as $N\to \infty$, of 
			\begin{equation*}
				\left( 2\sqrt{N} (\widehat{H}_{N,0}- H_{0}), 	2N ^{2-2H_{j}}(\widehat{H}_{N,j}- H_{j})\right)
			\end{equation*}
			to 
			\begin{equation*}
				\left( Z, c_{2, H_{j}}R ^{2H_{j}-1}_{1}, j=1,...,d\right),
			\end{equation*}
			$Z\sim N (0, \sigma _{p, H_{0}}^{2})$ and $Z$ is independent of $ R ^{2H_{j}-1}_{1}, j=1,...,d$.
		\end{enumerate}
		
	\end{remark}
	\subsection{Infinite chaos expansion }
	Let $(W(h), h\in H)$ be an isonormal process and let $(h_{i}, i\geq 1)$ be a family of elements of $H$ such that for every $i,j\geq 1$
	\begin{equation*}
		\langle h_{i}, h_{j} \rangle _{H}= \rho_{H} (i-j),
	\end{equation*}
	where $\rho_{H}$ is the auto-correlation function of the fractional noise given by (\ref{rho}).  Consider the sequence $(V_{N}, N\geq 1)$ given by 
	\begin{equation}\label{vn3}
		V_{N}= \frac{1}{\sqrt{N}} \sum_{k=1} ^{N} I_{p} (h_{k}^{\otimes p}).
	\end{equation} and let
	\begin{equation}
		\label{y}
		Y=e ^{W(h_{1})}=\sqrt{e}\sum_{n\geq 0} \frac{1}{n!}I_{n} (h_{1}^{ \otimes n}).
	\end{equation}
	Obviously $(V_{N}, N\geq 1)$ has the same finite-dimensional distribution as (\ref{xn}) (when $H=H_{0}$).  Assume 
	\begin{equation}
		\label{cond}
		0<H < 1-\frac{1}{2p}.
	\end{equation}
	By Theorem \ref{tt3},  if (\ref{cond}) holds true, then  $(V_{N}, N\geq 1)$ converges in law, as $ N\to \infty$, to $Z\sim N(0, \sigma _{p,H}^{2})$. Moreover, we have the following estimate for the Wasserstein distance (see \cite{NP1}): if $N$ is large, 
	\begin{equation}
		\label{5a-3}	d_{W}( V_{N}, Z) \leq C
		\begin{cases} n ^{-\frac{1}{2}}, \mbox{ if } H\in (0, \frac{1}{2}]\\
			n ^{H-1}, \mbox{ if } H\in [\frac{1}{2}, \frac{2p-3}{2p-2})\\
			n^{pH-p+\frac{1}{2}}, \mbox{ if } H \in [\frac{2p-3}{2p-2}, \frac{2p-1}{2p}).
		\end{cases}
	\end{equation}

	We  check the joint convergence  in law of the couple $(X_{N}, Y)$ when $N\to \infty$ and we evaluate the Wasserstein distance associated to it. 
	
	\begin{prop}
		Let $V_{N}, Y$ be given by (\ref{vn3}), (\ref{y}), respectively. Then
		\begin{equation*}
			(V_{N}, Y)\to ^{(d)} (Z, Y)
		\end{equation*}
		where $Z\sim N(0, \sigma _{p,H}^{2})$ is independent of $Y$. Moreover, for $N$ large
		\begin{equation}\label{14i-1}
			d_{W}( P_{(V_{N}, Y)}, P_{Z}\otimes P_{Y}) \leq C
			\begin{cases} n ^{-\frac{1}{2}}, \mbox{ if } H\in (0, \frac{1}{2}]\\
				n ^{H-1}, \mbox{ if } H\in [\frac{1}{2}, \frac{3}{4})\\
				n ^{H-1}+n^{pH-p+\frac{1}{2}}, \mbox{ if } H \in [\frac{3}{4}, \frac{2p-1}{2p}).
			\end{cases}
		\end{equation}
		
	\end{prop}
	\begin{proof} In order to get the joint convergence of $\left( (V_{N}, Y), N\geq 1\right)$, we need to check (\ref{cc6}).  We have 
		\begin{eqnarray*}
			\mathbf{E} (V_{N}Y)&=&\sqrt{e} \frac{1}{\sqrt{N}}\sum_{k=1}^{N} \mathbf{E} I_{p} (h_{k} ^{\otimes p})Y=\sqrt{e} \frac{1}{\sqrt{N}}\sum_{k=1}^{N} \mathbf{E} I_{p} (h_{k} ^{\otimes p})\frac{1}{p!} I_{p} ( h_{1} ^{\otimes p}) \\
			&=& \sqrt{e}\frac{1}{\sqrt{N}}\sum_{k=1}^{N} \langle h_{k}, h_{1} \rangle _{p} =\sqrt{e} \frac{1}{\sqrt{N}}\sum_{k=1}^{N} \rho_{H}(k-1) ^{p}. 
		\end{eqnarray*}
		By isolating the term with $k=1$, we have
		\begin{eqnarray*}
			\mathbf{E}( V_{N}Y)&=& \sqrt{e}\frac{1}{\sqrt{N}}\left( 1+ \sum_{k\geq 2} (k-1) ^{(2H-2)p}\right) \leq C \frac{1}{\sqrt{N}},
		\end{eqnarray*} 
		since the series $\sum_{k\geq 1} k^{(2H-2)p}$ is convergent due to (\ref{cond}).  Then, by Theorem \ref{tt2}, 
		\begin{equation}\label{5a-1}
			(V_{N}, Y) \to ^{(d)} _{N\to \infty} (Z, Y),
		\end{equation}
		where $ Z\sim N(0, \sigma_{p,H} ^{2})$ and $Z, Y$ are independent random variables.

		Let us evaluate the rate of convergence under the Wasserstein distance  for (\ref{5a-1}). We compute the quantity $\mathbf{E}\langle D(-L) ^{-1} V_{N}, DY\rangle _{H}^{2}$. We have

		\begin{equation*}
			D(-L) ^{-1} V_{N} = \frac{1}{ \sqrt{N}} \sum_{k=1} ^{N} I_{p-1} (h_{k}^{\otimes p-1}) h_{k}, \hskip0.3cm DY=Yh_{1}
		\end{equation*}
		and
		\begin{equation*}
			\langle D(-L) ^{-1} V_{N}, DY\rangle _{H}=  \frac{1}{ \sqrt{N}} \sum_{k=1} ^{N} I_{p-1} (h_{k}^{\otimes p-1}) Y \langle h_{k}, h_{1}\rangle _{H}.
		\end{equation*}
		Hence,
		\begin{eqnarray*}
			\mathbf{E}\langle D(-L) ^{-1} V_{N}, DY\rangle _{H}^{2}&=& \frac{1}{N} \sum_{k,l=1}^{N} I_{p-1} (h_{k}^{\otimes p-1})I_{p-1} (h_{l}^{\otimes p-1})Y ^{2} \langle h_{k}, h_{1}\rangle _{H}\langle h_{l}, h_{1}\rangle _{H}\\
			&=&  \frac{1}{N} \sum_{k,l=1}^{N} \sum_{r=0} ^{p-1} r! (C_{p-1}^{r}) ^{2} \mathbf{E} I_{2p-2r-2}\left(h_{k} ^{\otimes p-1} \otimes _{r} h_{l} ^{\otimes p-1}\right) Y^{2}\langle h_{k}, h_{1}\rangle _{H}\langle h_{l}, h_{1}\rangle _{H}, 
		\end{eqnarray*}
		where we applied the product formula (\ref{prod}). Since 
		\begin{equation*}
			Y^{2}= e ^{2W(h_{1})} =e \sum_{n\geq 0} \frac{ 2 ^{n}}{n!} I_{n} (h_{1} ^{\otimes n}),
		\end{equation*}
		we have, for $r=0,...,p-1$,
		\begin{eqnarray*}
			&&	\mathbf{E} I_{2p-2r-2}\left(h_{k} ^{\otimes p-1} \otimes _{r} h_{l} ^{\otimes p-1}\right) Y^{2}\\
			&=&e\frac{ 2 ^{2p-2r-2}}{(2p-2r-2)!} \mathbf{E}  I_{2p-2r-2}\left(h_{k} ^{\otimes p-1} \otimes _{r} h_{l} ^{\otimes p-1}\right) I_{2p-2r-2} ( h_{1} ^{\otimes 2r-2r-2})\\
			&=&e 2 ^{2p-2r-2} \langle (h_{k} ^{\otimes p-1} \widetilde{\otimes} _{r} h_{l} ^{\otimes p-1}, h_{1} ^{\otimes 2p-2r-2}\rangle _{ H ^{\otimes 2p-2r-2}}\\
			&=& e 2 ^{2p-2r-2} \langle h_{k}, h_{l} \rangle _{H}^{r}  \langle h_{k}, h_{1} \rangle _{H}^{p-r-1} \langle h_{l}, h_{1} \rangle _{H}^{p-r-1}.
		\end{eqnarray*}
		Consequently, 
		\begin{eqnarray*}
			\mathbf{E}\langle D(-L) ^{-1} V_{N}, DY\rangle _{H}^{2}
			&=&e \sum_{r=0} ^{p-1} r! (C_{p-1}^{r}) ^{2} 2 ^{2p-2r-2} T(r,p, N)
		\end{eqnarray*}
		with
		\begin{eqnarray}
			T(r,p,N)&=& \frac{1}{N} \sum_{k,l=1}^{N}  \langle h_{k}, h_{1} \rangle _{H}^{p-r} \langle h_{k}, h_{1} \rangle _{H}^{p-r} \langle h_{l}, h_{1} \rangle _{H}^{p-r}\nonumber \\
			&=&  \frac{1}{N} \sum_{k,l=1}^{N}  \rho_{H}(k-l) ^{r} \rho_{H} (k-1)^{p-r} \rho_{H}(l-1) ^{p-r}.\label{5a-2}
		\end{eqnarray}
		We now evaluate $	T(r,p,N)$ for $r=0,1,...,p-1$. We write 
		\begin{eqnarray*}
			T(r,p,N)&=& \frac{1}{N} \sum _{k=1} ^{N}\rho_{H}(k-1) ^{2(p-r)} + \frac{1}{N} \sum_{k,l=1; k\not=l} ^{N} \rho_{H}(k-l) ^{r} \rho_{H} (k-1)^{p-r} \rho_{H}(l-1) ^{p-r}\\
			&:=& T_{1}(r,p,N)+ T_{2}(r,p,N).
		\end{eqnarray*}
		Let us first treat the term $ T_{1}(r,p,N)$ with $r=0,1,..,p-1$. One has 
		\begin{eqnarray*}
			T_{1}(r,p,N)&=&\frac{1}{N} \left( 1+ \sum_{k\geq 2} \rho_{H}(k-1) ^{2(p-r)} \right) \leq C\frac{1}{N}\left( 1+ \sum_{k\geq 2} (k-1) ^{(2H-2) (2p-2r)}\right)\\
			&\leq & C\frac{1}{N} \left( 1+ \sum _{k\geq 1} k ^{2H-2}\right)\leq C\frac{1}{N} \left( 1+ N ^{2H-1}\right) \\
			&\leq & C\left( N ^{-1}+ N ^{2H-2}\right). 
		\end{eqnarray*}
		For $T_{2}(r,p,N)$, we can write
		
		\begin{eqnarray*}
			T_{2}(r,p,N)&=& 2\frac{1}{N} \sum_{k,l=1; k>l} ^{N} \rho_{H}(k-l) ^{r} \rho_{H} (k-1)^{p-r} \rho_{H}(l-1) ^{p-r}\\
			&\leq & C \frac{1}{N}\left( \sum_{k=2} ^{N}  \rho_{H}(k-l) ^{p} + \sum_{k>l\geq 2} (k-l) ^{(2H-2) r} (k-1) ^{(2H-2) (p-r)} (l-1) ^{(2H-2) (p-r)}\right)
		\end{eqnarray*}
		By (\ref{cond}), $ \sum_{k=2} ^{N}  \rho_{H}(k-l) ^{p} <\infty$ and so
		\begin{equation*}
			T_{2}(0,p,N)\leq C\frac{1}{N}\left( 1+\left(  \sum_{k\geq 2} (k-1) ^{(2H-2) p}\right) ^{2} \right) \leq C\frac{1}{N}
		\end{equation*}
		and for $r=1,...,p-1$, since $ (k-1)^{(2H-2)(p-r)}\leq (k-l)^{(2H-2)(p-r)}$, 
		\begin{eqnarray*}
			T_{2}(r,p,N) &\leq & C\frac{1}{N} \left( 1+ \sum _{k>l\geq 2} (k-l) ^{(2H-2)p} (l-1) ^{(2H-2)(p-r)}\right) \\
			&\leq & C\frac{1}{N} \left( 1+ \sum_{l=2}^{N} (l-1) ^{2H-2 } \sum_{k \geq 1} k ^{(2H-2)p}\right)\\
			&\leq & C\frac{1}{N} \left( 1+ N ^{2H-1}\right) \leq C (N ^{-1}+ N ^{2H-2}).
		\end{eqnarray*}
		From the above computations, we deduce that for $N$ sufficiently large,
		\begin{equation}\label{5a-4}
			\mathbf{E}\langle D(-L) ^{-1} V_{N}, DY\rangle _{H}^{2}\leq C (N ^{-1}+ N ^{2H-2}). 
		\end{equation}
		By combining (\ref{5a-3}) and (\ref{5a-4}), we get (\ref{14i-1}). 
		
	\end{proof}

	\subsection{Quantitative bounds in a central-noncentral  limit theorem}\label{sec5}
	
	Our approach allows to give qualitative bounds for the multidimensional sequences of multiple stochastic integral when only  one of these sequences converges to a normal distribution. Here we illustrate the method by treating a two -dimensional sequence in Wiener chaos, one component \\ being asymptotically Gaussian and the second component satisfying a non-central limit theorem. Such estimates are new in the literature and they cannot be obtained via the standard Stein method.   Let $ (B ^{H}_{t}, t\geq 0)$ be a fractional Brownian motion with Hurst index $H\in (0,1)$. For $N\geq 1$, define
	\begin{equation}
		\label{vn}
		V_{N}=q! \frac{1}{\sqrt{N}}\sum_{k=0} ^{N-1} H_{q} \left( B ^{H} _{k+1} - B ^{H} _{k}\right),  
	\end{equation}
	where $H_{q}$ is the Hermite polynomial of degree $q$. Then, the Breuer-Major theorem (see \cite{BM} or Theorem \ref{tt3}) states that, if $H\in \left(0, 1-\frac{1}{2q}\right) $ the sequence $ (V_{N}, N\geq 1)$ converges to a Gaussian random variable $Z\sim N(0, \sigma ^{2}_{q,H})$, where the variance $ \sigma ^{2}_{q,H} $ is explicily known. 
	
	On the other hand, the sequence $ (U_{N}, N\geq 1)$ given by 
	\begin{equation}
		\label{un}U_{N}= 2N ^{1-2H} \sum_{k=0} ^{N-1}H_{2} \left( B ^{H} _{k+1} - B ^{H} _{k}\right), \hskip03cm N\geq 1,
	\end{equation}
	converges in distribution, for $ H\in \left( \frac{3}{4}, 1\right)$, to $ c_{2,H} R ^{(2H-1)}$ where $ R ^{(2H-1)} $ is a Rosenblatt random variable with Hurst parameter $2H-1$ and again the constant $c_{2,H}>0$ is known. 
	
	Moreover, the random sequence $ (V_{N}, U_{N})$ converges in law, as $N\to \infty$, to $(Z, c_{2,H} R ^{(2H-1)})$, with $Z$ independent of $ R ^{(2H-1)}$. This can be obtained from the main findings in  \cite{NR} or \cite{NNP} but it also follows from our Theorem \ref{tt2}. The purpose is to find the rate of convergence, under the Wasserstein distance, for this two-dimensional limit theorem. 
	
	We have the following result. 
	\begin{prop}
		Let $V_{N}, U_{N}$ be given by (\ref{vn}, (\ref{un}), respectively. 
		Assume
		\begin{equation}\label{hyp}
			H\in \left( \frac{3}{4}, 1-\frac{1}{2q}\right) \Rightarrow q\geq 3. 
		\end{equation}
		Then
		\begin{equation*}
			(V_{N}, U_{N})\to ^{(d)} _{N \to \infty} (Z, c_{2,H} R ^{(2H-1)})
		\end{equation*}
		where $Z\sim N(0, \sigma ^{2}_{q,H})$ and $Z$ is independent from the Rosenblatt random variable $ R ^{(2H-1)}$.  Moreover
		\begin{equation}\label{27i-1}
			d_{W} \left( (V_{N}, U_{N}), (Z, c_{2,H} R ^{(2H-1)})\right) \leq c_{q, H} 
			\begin{cases}
				N ^{H-1}+ N ^{\frac{3}{2}-2H} \mbox { for } H\in \left( \frac{3}{4}, 1-\frac{1}{2(q-1)}\right)\\
				N ^{(H-1)q+\frac{1}{2}}+ N ^{\frac{3}{2}-2H}  \mbox{ for } \left( 1-\frac{1}{2(q-1)}, 1-\frac{1}{2q}\right).
			\end{cases}
		\end{equation}
		
	\end{prop}
	\begin{proof} The joint convergence of $((V_{N}, U_{N}), N\geq 1)$ is obtained via Proposition \ref{pp4}. By Theorem \ref{tt2}, we have
		\begin{eqnarray*}
			&&d_{W} \left( P_{(V_{N}, U_{N})}, P_{Z}\otimes  P_{c_{2,H} R ^{(2H-1)}}\right)\\
			&\leq& C\left[ \left(\mathbf{E} \left( \sigma ^{2}- \langle DV_{N}, D(-L) ^{-1} V_{N} \rangle\right) ^{2} \right) ^{\frac{1}{2}}+ d_{W} (P_{U_{N}}, P_{c_{2,H}R ^{(2H-1)}})+  \sqrt{ \mathbf{E} \left( \langle DV_{N}, DU_{N} \rangle \right) ^{2}}\right].
		\end{eqnarray*}
		
		We know the rate of convergence to their limits for each of the sequences $(V_{N}, N\geq 1)$ and $(U_{N}, N\geq 1)$.  If  one assumes (\ref{hyp}), then (see Theorem 4.1 in \cite{NP1})
		\begin{equation}
			\left(\mathbf{E} \left( \sigma ^{2}- \langle DV_{N}, D(-L) ^{-1} V_{N} \rangle\right) ^{2} \right) ^{\frac{1}{2}}\leq C_{H,q} \begin{cases}
				N ^{H-1} \mbox{ if } H \in \left( \frac{3}{4}, \frac{2q-3}{2q-2}\right]\\
				N ^{qH-q+\frac{1}{2}} \mbox{ if } H \in \left[ \frac{2q-3}{2q-2}, \frac{2q-1}{2q}\right). 
			\end{cases}\label{16n-1}
		\end{equation}
		Moreover, for any $H$ satisfying (\ref{hyp}) (see \cite{BN} or \cite{NP-book}, relation (7.4.13))
		\begin{equation}\label{16n-2}
			d_{W} (U_{N}, c_{2, H} R ^{(2H-1)} )\leq C_{H} N ^{\frac{3}{2}-2H}. 
		\end{equation}
		In particular, if $q=3$, it follows from (\ref{16n-1}) and (\ref{16n-2}) that
		\begin{eqnarray}
			&&	d_{W} (V_{N}, Z)+ 	d_{W} (U_{N}, c_{2, H} R ^{(2H-1)} )\leq C_{H}  \left( N ^{\frac{3}{2}-2H}+ N ^{3H-\frac{5}{2}}\right) \nonumber \\
			&&\leq C_{H} \begin{cases}
				N ^{\frac{3}{2}-2H} \mbox{ if } H \in \left(\frac{3}{4}, \frac{4}{5}\right)\\N ^{3H-\frac{5}{2}} \mbox{ if }H \in \left[ \frac{4}{5}, \frac{5}{6}\right) 
			\end{cases}\label{33}
		\end{eqnarray}
		Let us estimate the quantity $ \sqrt{ \mathbf{E} \left( \langle DV_{N}, DU_{N} \rangle \right) ^{2}}$.  Denote by $\mathcal{H}$ the canonical Hilbert space associated to the fractional Brownian motion, defined as the closure of the set of step functions on the positive  real line with respect to the scalar product 
		\begin{equation*}
			\langle 1_{[0,t]}, 1_{[0,s]}\rangle _{\mathcal{H}}= \mathbf{E} B ^{H} _{t} B ^{H}_{s}=\frac{1}{2}( t ^{2H}+ s^{2H}-\vert t-s\vert ^{2H}).
		\end{equation*}
		We can write, if $I_{q}$ is the multiple stochastic integral with respect to the isonormal process generated by $ B ^{H}$, 
		\begin{equation*}
			V_{N}= I_{q} (f_{N}) \mbox{ with } f_{N} = \frac{1}{\sqrt{N}} \sum_{k=1}^{N} h_{k} ^{\otimes q}
		\end{equation*}
		and
		\begin{equation*}
			U_{N}= I_{2}(g_{N}) \mbox{ with } g_{N}= N ^{1-2H} \sum_{l=1}^{N} h_{l} ^{\otimes 2},   
		\end{equation*}
		where $h_{k}= 1_{ [k-1, k)}$ for $k=1,..., N$. In particular $\Vert h_{k}\Vert _{\mathcal{H}}=1$ and
		\begin{equation}\label{2d-1}
			\langle h_{k}, h_{l}\rangle _{\mathcal{H}}= \rho_{H} (k-l)
		\end{equation}
		with $\rho_{H}$ from (\ref{rho}). Thus
		\begin{eqnarray*}
			\langle DV_{N}, DU_{N}\rangle  &=& 2q N ^{\frac{1}{2}-2H} \sum_{k,l=1}^{N} I_{q-1} (h_{k}^{ \otimes (q-1)}  I_{1} (h_{l}) \langle h_{k}, h_{l} \rangle \\
			&=&  2q N ^{\frac{1}{2}-2H} \sum_{k,l=1}^{N} \left[ I_{q} ( h_{k} ^{\otimes (q-1)} \otimes h_{l}) + (q-1) I_{q-2} ( h_{k} ^{\otimes (q-1)}\otimes _{1}  h_{l})\right]  \langle h_{k}, h_{l}\rangle  \\
			&=& 2q N ^{\frac{1}{2}-2H} \sum_{k,l=1}^{N} \left[ I_{q} ( h_{k} ^{\otimes (q-1)} \otimes h_{l}) + (q-1) I_{q-2} ( h_{k} ^{\otimes (q-2)})\langle h_{k}, h_{l}\rangle \right] \langle h_{k}, h_{l}\rangle,
		\end{eqnarray*}
		where we applied the product formula (\ref{prod}). Consequently,
		\begin{eqnarray*}
			&&	\mathbf{E}	\langle DV_{N}, DU_{N} \rangle ^{2}\\ 
			&\leq &c_{q}N ^{1-4H} \left[ \sum_{i,j,k,l=1}^{N} \langle h_{i} ^{\otimes (q-1)}\tilde{\otimes} h_{j},  h_{k} ^{\otimes (q-1)}\tilde{\otimes} h_{l}\rangle  \langle h_{i}, h_{j} \rangle  \langle h_{k}, h_{l} \rangle  + \langle h_{i}, h_{k} \rangle ^{q-2} \langle h_{i}, h_{j} \rangle ^{2}  \langle h_{k}, h_{l} \rangle ^{2} \right]\\
			&\leq & c_{q} N ^{1-4H}\left[  \sum_{i,j,k,l=1}^{N} \langle h_{i}, h_{k} \rangle ^{q-1} \langle h_{i}, h_{j} \rangle  \langle h_{k}, h_{l} \rangle  \langle h_{j}, h_{l} \rangle +\sum_{i,j,k,l=1}^{N} \langle h_{i}, h_{k} \rangle ^{q-2} \langle h_{i}, h_{j} \rangle  \langle h_{k}, h_{l} \rangle  \langle h_{i}, h_{l} \rangle \langle h_{j}, h_{k} \rangle \right. \\
			&&+\left.  \sum_{i,j,k,l=1}^{N} \langle h_{i}, h_{k} \rangle ^{q-2} \langle h_{i}, h_{j} \rangle ^{2}  \langle h_{k}, h_{l} \rangle ^{2} \right]=: a_{1,N}+ a_{2, N} + a_{3, N}.
		\end{eqnarray*}
		We used  Lemma 4.5 in \cite{T} in order to expres the scalar product $ \langle h_{i} ^{\otimes (q-1)}\tilde{\otimes} h_{j},  h_{k} ^{\otimes (q-1)}\tilde{\otimes} h_{l}\rangle $.  Using the inequality 
		$$\langle h_{i}, h_{j} \rangle  \langle h_{k}, h_{l} \rangle  \langle h_{i}, h_{l} \rangle \langle h_{j}, h_{k} \rangle\leq \frac{1}{2} \left( \langle h_{i}, h_{j} \rangle ^{2}  \langle h_{k}, h_{l} \rangle ^{2}+ \langle h_{i}, h_{l} \rangle ^{2}  \langle h_{k}, h_{j} \rangle ^{2}\right),$$ we get
		$a_{2,N}\leq  a_{3, N} $ so we have to estimate $ a_{1,N} $ and $ a_{3, N}$.  Now, by (\ref{2d-1}),
		\begin{eqnarray*}
			a_{3,N} &=& c_{q} N ^{1-4H} \sum_{i,j,k,l=1}^{N} \rho_{H} (i-k) ^{q-2} \rho_{H} (i-j)^{2} \rho _{ H}(k-l) ^{2} \\
			&\leq & c_{q} N ^{1-4H} \sum_{i,k=1} ^{N}  \rho_{H} (i-k) ^{q-2}\left( \sum _{a=-N}^{N} \rho_{H} (a) ^{2} \right) ^{2}.
		\end{eqnarray*}
		By using the bound  $ \sum _{a=-N}^{N} \rho_{H} (a)^{2}\leq c_{H} N ^{4H-3}$ we obtain 
		\begin{eqnarray*}
			a_{3,N}&\leq & c_{q,H} N ^{4H-5}  \sum_{i,k=1} ^{N}  \rho_{H} (i-k) ^{q-2} \leq c_{q,H}N ^{4H-4} \sum _{k\geq 1} k ^{(2H-2)(q-2)}\\
			&\leq &  c_{q,H}N ^{4H-4}\begin{cases}
				1, \mbox{ if } H<1-\frac{1}{2(q-2)}\\
				\log (N) \mbox{ if } H=1-\frac{1}{2(q-2)}\\
				N ^{(2H-2)(q-2)+1} \mbox{ if } H \in \left( 1-\frac{1}{2(q-2)}, 1-\frac{1}{2q}\right).
			\end{cases}
		\end{eqnarray*}
		For $q=3$, we have for $H\in \left( \frac{3}{4}, \frac{5}{6}\right)$,
		\begin{equation}\label{31}
			a_{3,N} \leq c_{H} N ^{6H-5}
		\end{equation}
		Let us deal with
		\begin{eqnarray*}
			a_{1,N} &= & c_{q,H} N ^{1-4H} \sum_{i,j,k,l=1}^{N}\rho_{H} (i-k) ^{q-1} \rho_{H} (i-j)\rho_{H} (k-l) \rho_{H} (j-l).
		\end{eqnarray*}
		This summand is the most complicated. Similar quantities (but not exactly the same!) have been treated in e.g. \cite{NP1}, proof of Theorem 4.1. We decompose  the sum over $(i,j,k,l)\in \{1,...,N\}^{4} $ upon the following cases: 
		\begin{enumerate}
			\item $(i=j=k=l)$, 
			\item $\left(   (i=j=k, l\not=i), (i=j=l, k\not=i), (i=k=l, j\not=i), (j=k=l, i\not= j)\right)$,
			\item   $\left(( i=j, k=l, k\not=i), (i=k, j=l, j\not=i), (i=l, j=k, j\not=i)\right)$,
			\item 
			\begin{eqnarray*}
				&& \left( (i=j, k\not=i, k\not=l, l\not=i), (i=k, j\not=i, j\not=l, k\not=l), (i=l, k\not=i, k\not=j, j\not=i),\right. \\
				&&	\left.(j=k, k\not=i, k\not=l, l\not=i), (j=l, k\not=i, k\not=l, j\not=i), (k=l, k\not=i,k\not=j,j\not=i)\right).
			\end{eqnarray*}
			\item $i,j,k,l$ are all different.
		\end{enumerate}
		We denote by $ a_{1, N} ^{(j)}, j=1,2,3,4,5$ the sum of all the terms from the groups 1.-5. defined above. The first of these terms can be easily estimated since 
		\begin{equation}
			\label{a1}
			a_{1,N} ^{(1)}=c_{q,H} N ^{1-4H}  \sum _{i=1}^{N} \rho_{H}(0) ^{q+2}= c_{q,H} N ^{2-4H}.
		\end{equation}
		For, the first sum from point 2. 
		
		$$c_{q,H} N ^{1-4H} \sum_{i,l=1}^{N} \rho_{H} (i-l) ^{2} \leq c_{q,H} N ^{2-4H} \sum_{i=1}^{N} i ^{4H-4} \leq  c_{q,H} N ^{2-4H} N ^{4H-3} =c_{q,H} N ^{-1}$$
		while the second from point 2.
		$$c_{q,H} N ^{1-4H} \sum_{i,k=1}^{N} \rho_{H}(i-k) ^{q} \leq c_{q,H} N ^{2-4H} \sum_{ k\in \mathbb{Z}} \rho_{H}(k)^{q} \leq c_{q,H} N ^{2-4H}.$$
		So, by symmetry,
		\begin{equation}
			\label{a2}
			a_{2,N} ^{(2)} \leq c_{q,H} ( N^{-1}+ N ^{2-4H}) \leq c_{q,H} N ^{-1}.
		\end{equation}
		The sums from group 3. are similar to the those from group 2. and we get 
		\begin{equation}
			\label{a3}
			a_{1,N} ^{(3)}\leq c_{q,H} N ^{-1}.
		\end{equation}
		Let us with the summands corresponding to point 4. The first one in this set reads 
		\begin{eqnarray*}
			&&	c_{q, H} N ^{1-4H} \sum_{ i\not=k\not=l\not=i}\rho_{H} (i-k) ^{q-1} \rho_{H} (k-l) \rho_{H} (i-l) \\
			&&\leq  c_{q} N ^{2-4H} \sum_{a,b=-N}^{N} \vert \rho_{H}\vert (a-b) ^{q-1} \vert \rho_{H}\vert (a)  \vert \rho_{H}\vert (b)  \leq  c_{q} N ^{2-4H} \sum_{a,b=-N}^{N} \vert  \rho_{H}\vert (a-b) ^{q-1} \vert \rho_{H}\vert (a) ^{2}\\
			&&\leq  c_{q,H} N ^{2-4H} \sum_{a=-N}^{N} \vert a\vert ^{4H-4}\sum_{b=-2N}^{2N} \vert b\vert ^{(2H-2) (q-1)}.
		\end{eqnarray*}
		It follows that this term is less than
		$$c_{q,H}\begin{cases}  N ^{-1} \mbox{ if } H < 1-\frac{1}{2(q-1)}\\
			N ^{-1} \log N \mbox{ if } H=1-\frac{1}{2(q-1)}\\
			N ^{(2H-2) (q-1)+2} \mbox{ if } H \in \left( 1-\frac{1}{2(q-1)}, 1-\frac{1}{2q}\right).
		\end{cases}
		$$
		Regarding the second summant in 4., we can bound as follows
		\begin{eqnarray*}
			&&	c_{q,H} N ^{1-4H}\sum_{i\not=j\not=l\not=i} \rho_{H} (i-j) \rho_{H} (i-l) \rho_{H}(j-l)\\
			&&\leq c_{q,H} N ^{1-4H} N^{3} N ^{6H-6} \frac{1}{ N ^{3}}\sum_{i\not=j\not=l\not=i}\left( \frac{ \vert i-j\vert}{N} \right) ^{2H-2} \left(\frac{ \vert i-l\vert}{N} \right) ^{2H-2} \left(\frac{ \vert j-l\vert}{N} \right) ^{2H-2} \\
			&=& c_{q,H} N ^{2H-2}  \frac{1}{ N ^{3}}\sum_{i\not=j\not=l\not=i}\left( \frac{ \vert i-j\vert}{N} \right) ^{2H-2} \left(\frac{ \vert i-l\vert}{N} \right) ^{2H-2} \left(\frac{ \vert j-l\vert}{N} \right) ^{2H-2} \leq  c_{q,H} N ^{2H-2},
		\end{eqnarray*}
		since the quantity $ \frac{1}{ N ^{3}}\sum_{i\not=j\not=l\not=i}\left( \frac{ \vert i-j\vert}{N} \right) ^{2H-2} \left(\frac{ \vert i-l\vert}{N} \right) ^{2H-2} \left(\frac{ \vert j-l\vert}{N} \right) ^{2H-2} $ is a Riemann sum that converges to $\int_{[0,1]^{3}} \vert x-y\vert ^{2H-2} \vert y-z\vert ^{2H-2} \vert z-x\vert ^{2H-2}dxdydz <\infty$. We have similar bounds for the other terms and we get
		\begin{equation}\label{a4}
			a_{1, N} ^{(4)} \leq c_{q, H} N ^{2H-2}.
		\end{equation}
		Notice that the estimation of the dominant term, the second in this group is sharp. 
		
		For the only summand in group 5., we separate its analysis uopon all the possible orders: $i>j>k>l, i>j>l>k,......$. The first summand is treated as follows 
		\begin{eqnarray*}
			&& c_{q} N ^{1-4H}\sum_{i>j>k>l} \rho_{H} (i-k) ^{q-1} \rho_{H} (i-j)\rho_{H} (k-l) \rho_{H} (j-l)\\
			&\leq & c_{q,H} N ^{1-4H}\sum_{i>j>k>l} \vert i-k\vert ^{2H-2)(q-1)}\vert i-j\vert ^{2H-2} \vert k-l\vert ^{2H-2} \vert j-l\vert ^{2H-2}\\
			&\leq & c_{q,H} N ^{1-4H}\sum_{i>j>k>l} \vert i-k\vert ^{(2H-2)(q-1)}\vert i-j\vert ^{2H-2} \vert k-l\vert ^{4H-4}\\
			&\leq &  c_{q,H} N ^{1-4H}\sum_{i>j>k} \vert i-k\vert ^{(2H-2)(q-1)}\vert i-j\vert ^{2H-2} \sum_{ l=-N} ^{N} \vert l\vert ^{4H-4}\\
			&\leq & c_{q,H} N ^{-2} \sum_{i>j>k} \vert i-k\vert ^{(2H-2)(q-1)}\vert i-j\vert ^{2H-2}\\
			&\leq &c_{q,H} N ^{-2} \sum_{i>k} \vert i-k\vert ^{(2H-2)(q-1)}\sum_{ j=-N} ^{N} \vert j\vert ^{2H-2} \leq c_{q,H} N ^{2H-3}\sum_{i>k} \vert i-k\vert ^{(2H-2)(q-1)}\\
			&\leq & c_{q,H} N ^{2H-2} \sum_{k=1}^{N}k ^{(2H-2)(q-1)}.
		\end{eqnarray*}
		With analogous estimates for the other cases of point 5., we obtain 
		\begin{equation}\label{a5}
			a_{1,N}^{(5)}\leq c_{q,H} \begin{cases}
				N ^{2H-2} \mbox{ if } H< 1-\frac{1}{2(q-1)}\\
				N ^{2H-2} \log N \mbox{ if }=   1-\frac{1}{2(q-1)}\\
				N ^{(2H-2)q+1}\mbox{ if } H\in \left(  1-\frac{1}{2(q-1)},  1-\frac{1}{2q}\right).
			\end{cases}
		\end{equation}
		So, by (\ref{a1}), (\ref{a2}), (\ref{a3}), (\ref{a4}) and (\ref{a5})
		\begin{eqnarray*}
			a_{1,N}
			&\leq& c_{q,H} \begin{cases}
				N ^{2H-2} \mbox{ if } H \in \left(\frac{3}{4}, 1-\frac{1}{2(q-1)}\right)\\
				N ^{(2H-2)q+1} \mbox{ if } H\in \left( 1-\frac{1}{2(q-1)}, 1-\frac{1}{2q}\right).
			\end{cases}
		\end{eqnarray*}
		Thus
		\begin{equation}\label{3d-1}
			\mathbf{E}	\langle DV_{N}, DU_{N} \rangle ^{2}\leq c_{q,H} \begin{cases}
				N ^{2H-2} \mbox{ if } H \in \left(\frac{3}{4}, 1-\frac{1}{2(q-1)}\right)\\
				N ^{(2H-2)q+1} \mbox{ if } H\in \left( 1-\frac{1}{2(q-1)}, 1-\frac{1}{2q}\right),
			\end{cases},
		\end{equation}
		the bound on the first branch being immaterial for $q=3,4$.  If $q=3$, then 
		\begin{equation}\label{32}
			\mathbf{E}	\langle DV_{N}, DU_{N} \rangle ^{2}\leq c_{H}N ^{6H-5}.
		\end{equation}
		We then obtain (\ref{27i-1}). 
		
	\end{proof} \qed

	\begin{remark}\label{rem4}
		\begin{enumerate} 
			\item For $q=3$, we have from (\ref{33}), (\ref{31}) and (\ref{32}),
			\begin{equation}\label{27i-2}
				d_{W} \left( (V_{N}, U_{N}), (Z, c_{2,H} R ^{(2H-1)})\right) \leq C_{H}\begin{cases}
					N ^{\frac{3}{2}-2H} \mbox{ if } H \in \left(\frac{3}{4}, \frac{4}{5}\right)\\N ^{3H-\frac{5}{2}} \mbox{ if }H \in \left[ \frac{4}{5}, \frac{5}{6}\right).
				\end{cases}
			\end{equation}
			
			\item 	It follows from the above calculation that the quantity $ \left( \mathbf{E} \langle DV_{N}, DU_{N} \rangle ^{2} \right) ^{ \frac{1}{2}}$, which somehow measures the correlation between $ V_{N} $ and $ U_{N}$ has the same size, for $N$ large, as $	d_{W} (V_{N}, Z)$ (compare (\ref{16n-1}) and (\ref{3d-1})).
			
			\item A quantitative bound for the above limit theorem can be also obtained by using the estimate (\ref{10a-2}) in Remark \ref{rem3}.  Notice that (\ref{10a-2})gives 
			\begin{equation*}
				\mathbf{E} \langle DV_{N}, DU_{N} \rangle ^{2} \leq C_{H} \mathbf{E} \left( \Vert f_{N} \otimes _{1} f_{N} \Vert +  \Vert f_{N} \otimes _{2} f_{N} \Vert \right). 
			\end{equation*}
			By using the calculations in the proof of Theorem 4.1 in \cite{NP1} and since $ \mathbf{E}G_{N} \leq C_{H} $ (with $C_{H}>0$ not depending on $N$), we get 
			\begin{equation*}
				\mathbf{E} \langle DV_{N}, DU_{N} \rangle ^{2}\leq C_{H} \left( N ^{-\frac{1}{2}}+ N ^{H-1}+ N ^{1-q(1-H)}\right),
			\end{equation*}
			which is in general less good than (\ref{27i-1}). For instance, if $q=3$, we have 
			\begin{equation*}
				\mathbf{E} \langle DV_{N}, DU_{N} \rangle ^{2}\leq C_{H}\left( N ^{-\frac{1}{2}}+ N ^{H-1}+ N ^{ 3H-2}\right),
			\end{equation*}
			and leads, for $ H\in \left( \frac{3}{4}, \frac{5}{6}\right)$, to 
			\begin{equation*}
				d_{W} \left( (V_{N}, U_{N}), (Z, c_{2,H} R ^{(2H-1)})\right) \leq C_{H}N ^{ \frac{3H}{2}-1},
			\end{equation*}
			which clearly is less optimal than (\ref{27i-2}).
		\end{enumerate}
	\end{remark}
	
	\subsection{The evolution of the solution to a semilinear stochastic  equation}
	The theory developed in Section \ref{sec2} can also be applied to quantify the evolution of a stochastic system defined by a stochastic differential equation. We present here a very simple example (a more complex situation, in the KPZ context, has been treated in \cite{Pi}). Let $ \lambda \in \mathbb{R}$ and consider the stochastic equation
	\begin{equation}
		\label{29i-1}
		X^{\lambda}_{t} = X_{0} + \lambda \int_{0} ^{t} b(X^{\lambda} _{s})ds + W_{t}, \hskip0.5cm t\geq 0
	\end{equation} 
	where $(W_{t}, t\geq 0)$ is a Wiener process. We assume that the drift $b: \mathbb{R}\to \mathbb{R}$ is differentiable and satisfies $\vert b'(x)\vert \leq M$ for every $x\in \mathbb{R}$. Then (\ref{29i-1}) admits a unique solution which is Malliavin differentiable and (see e.g. Exercice 2.2.1 in \cite{N}) for $a<t$, 
	\begin{equation*}
		D_{a} X ^{\lambda }_{t}= e ^{ \int_{a} ^{t} b'(X^{\lambda}_{s})ds}. 
	\end{equation*}
	The solution to (\ref{29i-1}) is a Gaussian process for $\lambda =0$ and for $\lambda \not=0$, its law is non-Gaussian if $b$ is nonlinear. Theorem \ref{tt1} allows to quantify the dependence structure between the components of the vector $ ( X _{t}^{\lambda}, X_{t}^{0})$ at each time $t>0$.  Indeed, by Theorem \ref{tt1},
	\begin{eqnarray*}
		&&	d_{W} \left( P_{( X _{t}^{\lambda}, X_{t}^{0}) }, P_{ X ^{
				\lambda}_{t}}\otimes P_{ X^{0}_{t}}\right) \leq C \int_{0} ^{t} D_{a} X_{t} ^{\lambda} da  \\
		&&\leq  C \int_{0} ^{t} e ^{ \int_{a} ^{t} b'(X^{\lambda}_{s})ds}da \leq C\int_{0} ^{t} e ^{\lambda M (t-\lambda)}= \frac{C}{M\lambda} ( e ^{M\lambda t}-1):= g(\lambda).
	\end{eqnarray*}
	The function $g$ provides a quantitative estimate for the dependence between $ X^{\lambda}$ and $X ^{0}$ for any $\lambda$, at any time. This function converges to a constant when $\lambda \to 0$ and to infinity as $\lambda \to \infty$. When $\lambda$ tends to $-\infty$, $g(\lambda)$ converges to zero, i.e.  the drift forces  the solution to (\ref{29i-1}) to be independent of the noise at each time.

	\section{Appendix: Wiener-Chaos and Malliavin derivatives}

	Here we describe the elements from stochastic analysis that we will need in the paper. Consider $H$ a real separable Hilbert space and $(W(h), h \in H)$ an isonormal Gaussian process on a probability space $(\Omega, {\cal{A}}, P)$, which is a centered Gaussian family of random variables such that ${\bf E}\left[ W(\varphi) W(\psi) \right]  = \langle\varphi, \psi\rangle_{H}$. Denote by  $I_{n}$ the multiple stochastic integral with respect to
	$B$ (see \cite{N}). This mapping $I_{n}$ is actually an isometry between the Hilbert space $H^{\odot n}$(symmetric tensor product) equipped with the scaled norm $\frac{1}{\sqrt{n!}}\Vert\cdot\Vert_{H^{\otimes n}}$ and the Wiener chaos of order $n$ which is defined as the closed linear span of the random variables $H_{n}(W(h))$ where $h \in H, \|h\|_{H}=1$ and $H_{n}$ is the Hermite polynomial of degree $n \in {\mathbb N}$
	\begin{equation*}
		H_{n}(x)=\frac{(-1)^{n}}{n!} \exp \left( \frac{x^{2}}{2} \right)
		\frac{d^{n}}{dx^{n}}\left( \exp \left( -\frac{x^{2}}{2}\right)
		\right), \hskip0.5cm x\in \mathbb{R}.
	\end{equation*}
	The isometry of multiple integrals can be written as follows: for $m,n$ positive integers,
	\begin{eqnarray}
		\mathbf{E}\left(I_{n}(f) I_{m}(g) \right) &=& n! \langle \tilde{f},\tilde{g}\rangle _{H^{\otimes n}}\quad \mbox{if } m=n,\nonumber \\
		\mathbf{E}\left(I_{n}(f) I_{m}(g) \right) &= & 0\quad \mbox{if } m\not=n.\label{iso}
	\end{eqnarray}
	It also holds that
	\begin{equation*}
		I_{n}(f) = I_{n}\big( \tilde{f}\big)
	\end{equation*}
	where $\tilde{f} $ denotes the symmetrization of $f$ defined by the formula \\ 
	$$\tilde{f}
	(x_{1}, \ldots , x_{n}) =\frac{1}{n!} \sum_{\sigma \in {\cal S}_{n}}
	f(x_{\sigma (1) }, \ldots , x_{\sigma (n) } ).$$
	\\\\
	We recall that any square integrable random variable which is measurable with respect to the $\sigma$-algebra generated by $W$ can be expanded into an orthogonal sum of multiple stochastic integrals
	\begin{equation}
		\label{sum1} F=\sum_{n=0}^\infty I_{n}(f_{n})
	\end{equation}
	where $f_{n}\in H^{\odot n}$ are (uniquely determined)
	symmetric functions and $I_{0}(f_{0})=\mathbf{E}\left[  F\right]$.
	\\\\
	Let $L$ be the Ornstein-Uhlenbeck operator
	\begin{equation*}
		LF=-\sum_{n\geq 0} nI_{n}(f_{n})
	\end{equation*}
	if $F$ is given by (\ref{sum1}) and it is such that $\sum_{n=1}^{\infty} n^{2}n! \Vert f_{n} \Vert ^{2} _{{\cal{H}}^{\otimes n}}<\infty$.
	\\\\
	For $p>1$ and $\alpha \in \mathbb{R}$ we introduce the Sobolev-Watanabe space $\mathbb{D}^{\alpha ,p }$  as the closure of
	the set of polynomial random variables with respect to the norm
	\begin{equation*}
		\Vert F\Vert _{\alpha , p} =\Vert (I -L) ^{\frac{\alpha }{2}} F \Vert_{L^{p} (\Omega )}
	\end{equation*}
	where $I$ represents the identity. We denote by $D$  the Malliavin  derivative operator that acts on smooth functions of the form $F=g(W(h_1), \dots , W(h_n))$ ($g$ is a smooth function with compact support and $h_i \in H$)
	\begin{equation*}
		DF=\sum_{i=1}^{n}\frac{\partial g}{\partial x_{i}}(W(h_1), \ldots , W(h_n)) h_{i}.
	\end{equation*}
	The operator $D$ is continuous from $\mathbb{D}^{\alpha , p} $ into $\mathbb{D} ^{\alpha -1, p} \left( H\right).$ The adjoint of $D$ is the divergence integral, denoted by $\delta$. It acts from $\mathbb{D} ^{\alpha -1, p} \left( H\right)$ onto $\mathbb{D}^{\alpha , p} $.

	We will intensively use the product formula for multiple integrals.
	It is well-known that for $f\in H^{\odot n}$ and $g\in H^{\odot m}$
	\begin{equation}\label{prod}
		I_n(f)I_m(g)= \sum _{r=0}^{n\wedge m} r! \left( \begin{array}{c} n\\r\end{array}\right) \left( \begin{array}{c} m\\r\end{array}\right) I_{m+n-2r}(f\otimes _r g)
	\end{equation}
	where $f\otimes _r g$ means the $r$-contraction of $f$ and $g$ (see e.g. Section 1.1.2 in \cite{N}). This contraction is defined, when $ H= L^{2}(T, \mathbb{B}, \nu)$ (where $\nu$ is a sigma-finite measure without atoms)
	\begin{eqnarray}
	&&	\label{contra}
		(f \otimes _{r} g) (t_{1},...,t_{n+m-2r})\\
		&=& \int_{T^{r}} f(u_{1},...,u_{r}, t_{1},...,t_{n-r})g(u_{1},...,u_{r}, t_{n-r+1},...,t_{n+m-2r})du_{1}....du_{r},\nonumber
	\end{eqnarray}
	for $r=1,..., n\wedge m$ and $f\otimes _{0}g=f\otimes g$, the tensor product. It holds that $f\otimes _{r}g\in H ^{\otimes n+m-2r}= L^{2}(T ^{n+m-2r})$. In general, the contraction $f\otimes _{r}g$ is not symmetric and we denote by $f\widetilde{\otimes } _{r} g$ its symmetrization. 
	

\end{document}